\documentclass[11pt,reqno]{amsart}

\usepackage[margin=1in]{geometry}
\usepackage{enumerate}
\usepackage[hidelinks]{hyperref}
\usepackage{amssymb}
\usepackage{mathrsfs}
\usepackage{graphicx}
\usepackage{amsthm}
\usepackage{amsmath}
\usepackage{bbm}
\usepackage{tikz}
\usetikzlibrary{decorations.pathreplacing}

\newtheorem{definition}{Definition}[section]
\newtheorem{theorem}[definition]{Theorem}
\newtheorem{lemma}[definition]{Lemma}
\newtheorem{corollary}[definition]{Corollary}
\newtheorem{remark}[definition]{Remark}
\newtheorem{example}[definition]{Example}
\newtheorem{conjecture}[definition]{Conjecture}

\newtheorem{proposition}[definition]{Proposition}

\def\makeCal#1{%
	\expandafter\newcommand\csname c#1\endcsname{\mathcal{#1}}}
\def\makeBf#1{%
	\expandafter\newcommand\csname b#1\endcsname{\mathbf{#1}}}
\def\makeBb#1{%
	\expandafter\newcommand\csname m#1\endcsname{\mathbb{#1}}}
\def\makeFrak#1{%
	\expandafter\newcommand\csname f#1\endcsname{\mathfrak{#1}}}
\def\makeScr#1{%
	\expandafter\newcommand\csname s#1\endcsname{\mathscr{#1}}}

\makeatletter
\count@=0
\loop
\advance\count@ 1
\edef\y{\@Alph\count@}%
\expandafter\makeCal\y
\expandafter\makeBf\y
\expandafter\makeBb\y
\expandafter\makeFrak\y
\expandafter\makeScr\y
\ifnum\count@<26
\repeat
\makeatother

\DeclareMathOperator{\Cat}{Cat}
\DeclareMathOperator{\End}{End}
\DeclareMathOperator{\diag}{diag}
\newcommand{\mone}{\mathbbm{1}}
\newcommand{\tG}{\tilde{G}}
\newcommand{\tDelta}{\tilde{\Delta}}
\newcommand{\hR}{\widehat{R}}

\usepackage{xy}
\renewcommand{\cH}{\mathcal{H}}

\newcommand\blfootnote[1]{%
	\begingroup
	\renewcommand\thefootnote{}\footnote{#1}%
	\addtocounter{footnote}{-1}%
	\endgroup
}

\allowdisplaybreaks

\title{Freidel-Maillet type equations on fused K-matrices over the positive part of $U_q(\widehat{\mathfrak{sl}}_2)$}
\author{Chenwei Ruan}
\date{}

\begin{document}
	\maketitle
	
	\begin{abstract}
		The positive part $U_q^+$ of the quantized enveloping algebra $U_q(\widehat{\mathfrak{sl}}_2)$ has a reflection equation presentation of Freidel-Maillet type (Baseilhac 2022). Its defining K-matrix has size $2 \times 2$ and can be expressed, using Rosso's embedding of $U_q^+$ into a $q$-shuffle algebra, as generating functions whose coefficients are Terwilliger's alternating PBW basis elements. This and older PBW bases of $U_q^+$ due to Damiani and Beck are unified by linear combinations of Catalan words (Ruan 2025). In this paper, we use this unification to define K-operators of any dimension $\geq 2$ in matrix form; the entries are explicit generating functions over $U_q^+$. Our main result is that any pair of such K-matrices, possibly of different dimensions, satisfy a Freidel-Maillet type equation. This yields a family of algebraic relations over $U_q^+$ that generalize Baseilhac's equation and can be used to study integrable systems or higher-spin representations. We conjecture that, for every spin, these K-matrices coincide with Lusztig's quasi R-matrix evaluated with its first leg on the corresponding finite-dimensional irreducible module. 
	\end{abstract}
	
	\tableofcontents
	
	\section{Introduction}
	\noindent The Yang-Baxter equation and the boundary Yang-Baxter equation appear in quantum integrable systems e.g.\ \cite{BS,cherednik1,faddeev,RSV1}, representation theory e.g.\ \cite{EFK,KKKO,KRS}, and geometry e.g.\ \cite{GKS,MO,nakajima}. The boundary Yang-Baxter equation is also known as the reflection equation. 
	
	\medskip
	\noindent The original motivation for quantum groups was a representation-theoretic framework for R-matrices \cite{drinfeld,jimbo,lusztig2}, i.e. solutions of the Yang-Baxter equation, where typically the quantum group is constructed as a deformation of the universal enveloping algebra of a suitable Lie algebra. Similarly, many examples of K-matrices (solutions of the reflection equation) are known to be related to quantum symmetric pairs \cite{CGM,kolb,noumi,NS}, which are quantization of classical symmetric pairs. In particular, this includes a coideal subalgebra of the quantum group that deforms the universal enveloping algebra of the fixed-point subalgebra of an involution. 
	
	\medskip
	\noindent Among the quantum groups arising in this way, the quantum affine algebras $U_q(\widehat{\mathfrak{g}})$ play a distinguished role: they govern integrable systems \cite{baxter,faddeev,FR,jimbo,sklyanin}, their representation theory is by now highly developed \cite{BCP,CP,jantzen,leclerc,lusztig2}, they are closely tied to combinatorics through quantum shuffles, words, and PBW bases \cite{green,rosso2,dalt,ter_alternating,ter_catalan}, and they admit rich presentations and realizations \cite{BS,baseilhac1,baseilhac2,beck,damiani,ter_qOns2}. In this paper, we focus on the rank one case $U_q(\widehat{\mathfrak{sl}}_2)$ which is the simplest and most computationally accessible, yet exhibiting remarkably rich structure. The algebra $U_q(\widehat{\mathfrak{sl}}_2)$ has a subalgebra $U_q^+$, called the positive part \cite{CP,lusztig2}. Both $U_q(\widehat{\mathfrak{sl}}_2)$ and $U_q^+$ are associative, noncommutative, and infinite-dimensional. 
	
	\medskip
	\noindent The algebra $U_q^+$ has a Hall algebra structure \cite{ringel}. The canonical basis and the dual canonical basis for $U_q^+$ were obtained in \cite{lusztig1} and in \cite{leclerc} respectively. Further, finite-dimensional irreducible representation of $U_q(\widehat{\mathfrak{sl}}_2)$ remain irreducible upon restriction to $U_q^+$, see \cite[Theorem 6.2.4]{AV25}, which goes back to results about the quantum Borel subalgebra \cite{BT,bowman,CG,HJ}. 
	
	\medskip
	\noindent In this paper, we are interested in the matrix solutions to the Yang-Baxter equation and a generalized form of the reflection equation. More precisely, an R-matrix is an $\End(\mC^2 \otimes \mC^2)$-valued formal Laurent series $R(t)$ which satisfies the \textit{Yang-Baxter equation} \cite{baxter,yang}
	\begin{equation}\label{eq:YBequation}
		R_{12}(t_1/t_2)R_{13}(t_1/t_3)R_{23}(t_2/t_3)=R_{23}(t_2/t_3)R_{13}(t_1/t_3)R_{12}(t_1/t_2). 
	\end{equation}
	
	\noindent Here we interpret $R_{12}(t)$ as $R(t) \otimes \mI$, where $\mI$ is the identity in $\End(\mC^2)$. We interpret $R_{13}(t)$ and $R_{23}(t)$ in a similar way. 
	
	\medskip
	\noindent In this paper, given an R-matrix $R(t)$, an $\widehat{\text{R}}$-matrix is an $\End(\mC^2 \otimes \mC^2)$-valued formal Laurent series $\hR(t)$ which satisfies 
	\begin{equation*}
		R_{12}(t_1/t_2)\hR_{13}(t_1/t_3)\hR_{23}(t_2/t_3)=\hR_{23}(t_2/t_3)\hR_{13}(t_1/t_3)R_{12}(t_1/t_2). 
	\end{equation*}
	
	\noindent Given an R-matrix and a $\widehat{\text{R}}$-matrix, a K-matrix is an $U_q^+ \otimes \End(\mC^2)$-valued formal Laurent series $K(t)$ which satisfies the \textit{Freidel-Maillet type equation} \cite{cherednik2,FM,KS}
	\begin{equation}\label{eq:reflectionequation}
		R(t_2/t_1)K_1(t_1)\hR(t_1t_2)K_2(t_2)=K_2(t_2)\hR(t_1t_2)K_1(t_1)R(t_2/t_1). 
	\end{equation}
	
	\noindent
	In the quantum inverse scattering method \cite{cherednik1,KS,sklyanin}, a K-matrix is the boundary reflection operator of an open spin chain, and the reflection equation is the condition ruling the commutativity of its transfer matrices, and hence ruling integrability \cite{sklyanin}. 
	
	\medskip
	\noindent In \cite{baseilhac1}, Pascal Baseilhac obtained a presentation for $U_q^+$ using a Freidel-Maillet type equation. This equation involves an R-matrix, a $\widehat{\text{R}}$-matrix with scalar entries, and a K-matrix. 
	
	\medskip
	\noindent In this paper, we will generalize the notions of the R-, $\widehat{\text{R}}$-, and K-matrices in terms of the underlying field and in terms of the dimension. We will obtain a generalized version of the Freidel-Maillet type equation from \cite{baseilhac1}. We remark that, in the terminology of integrable systems, our main theorem states: for every pair of spins $j_1,j_2$, the fused K-matrices $K^{(j_1)}$ and $K^{(j_2)}$ are the boundary reflection operators of a chain with bulk scattering operators $R^{(j_1,j_2)}$ and $\hR^{(j_1,j_2)}$. We remark that our approach is motivated by but logically independent of \cite{baseilhac1}. 
	
	\medskip
	\noindent Our result admits a conjectural interpretation in terms of the quasi R-matrix of Lusztig's \cite{lusztig2}. In Appendix D, we will see that a Freidel-Maillet type equation of the same form as \eqref{eq:reflectionequation} is produced by the completed element $K=(1\otimes\xi^{-1})\widetilde{R}$, where $\widetilde{R}$ is the quasi R-matrix and $\xi$ a certain Cartan-type element. In view of this, we conjecture that the K-matrix of \cite{baseilhac1}, and its higher-spin generalization $K^{(j)}(t)$ constructed in this paper, are Cartan corrections of the quasi R-matrix with the first leg acting on a finite-dimensional irreducible module; equivalently, it is obtained from the universal K-matrix associated with the q-Onsager algebra in the framework of tensor K-matrices \cite{AV24,LBG} by a suitable parameter specialization; see Remark \ref{rm:D1} and Conjecture \ref{conj:D2}. Moreover, after the grading shift of \cite{drinfeld2}, the R- and $\widehat{\text{R}}$-matrices of Definitions \ref{def:R}, \ref{def:hR} match, up to scalar multiples and sign convention, the evaluations of the universal R-matrix and of $q^{-\Omega}$ on tensor products of finite-dimensional irreducible modules; see Appendix D. We also note that \eqref{eq:reflectionequation} is not a $\psi$-twisted reflection equation in the axiomatic framework for tensor K-matrices \cite{AV24,AV22,LBG}; generalizing this framework so as to cover \eqref{eq:reflectionequation} and other degenerate examples is work in progress. 
	
	\medskip
	\noindent Explicit $L$-operators with one leg on the two-dimensional irreducible module, equivalently the matrices $R^{(\frac{1}{2},j)}(t)$ after projecting the other leg to a finite-dimensional quotient, are known; see \cite{BGKNR} and \cite[Section 2.5]{RSV2}, where explicit $L$-operators with an oscillator auxiliary space are also given. For spin $\frac{1}{2}$, the non-diagonal K-matrix of the open XXZ chain is known in closed form \cite{dVGR}. For higher spins, K-matrices are usually constructed recursively or by fusion \cite{DN,karowski,KR,MN}, and explicit closed formulas valid for all spins are uncommon. Likewise, for $j_1,j_2 \geq 1$ the R-matrices $R^{(j_1,j_2)}(t)$ are obtained recursively by fusion \cite{karowski,KR,MN,LBG} (see Definition \ref{def:R}), and no closed formula is known. In contrast, the K-matrices $K^{(j)}(t)$ of Definition \ref{def:K} are given in closed form for every $j \in \frac{1}{2}\mN^+$, and by Conjecture \ref{conj:D2} they are Cartan corrections of the quasi R-matrix with one leg on the corresponding irreducible module. 
	
	\medskip
	\noindent Our closed formulas also have consequences beyond the Freidel-Maillet equation. By \cite{ter_qOns1} the q-Onsager algebra has an injection into the algebra $\square_q$, obtained by gluing two copies of $U_q^+$, and the images of the q-Onsager generators are built from the same type of alternating generating functions that form the entries of the K-matrix $K^{(\frac{1}{2})}(t)$. Thus our closed formulas yield closed formulas for these images. Moreover, our fusion relation matches that of \cite{LBG}, whose fused K-operators are valued in the alternating central extension of the q-Onsager algebra \cite{ter_qOns2} and are obtained recursively; hence our result aids in supplying closed forms for the fused reflection operators of the algebraic Bethe ansatz \cite{RSV1,RSV2}. 
	
	\medskip
	\noindent We now introduce our tool of study. In \cite{rosso1,rosso2}, Rosso constructed an embedding of $U_q^+$ into a $q$-shuffle algebra. In \cite{ter_alternating,ter_catalan,ter_beck}, Terwilliger used the Rosso embedding to obtain closed form for two PBW bases for $U_q^+$ due to Damiani \cite{damiani} and Beck \cite{beck}. He also used the Rosso embedding to obtain the alternating PBW basis for $U_q^+$. As we will see, a closed form for the K-matrix can be obtained using the alternating PBW basis. In \cite{uniform}, we obtained a uniform approach to the three PBW bases mentioned above. In this paper, we will use the uniform approach to obtain a Freidel-Maillet type equation where the R-, $\widehat{\text{R}}$-, and K-matrices are of arbitrary meaningful dimensions. Moreover, the K-matrix can be written in closed form. 
		
	\medskip
	\noindent In order to construct R-, $\widehat{\text{R}}$-, and K-matrices of arbitrary dimension, we will use a fusion technique. In the literature, fusion techniques have been developed to construct fused R-matrices and K-matrices of higher dimension; see e.g.\ \cite{karowski,KR} for the fused R-matrices and e.g.\ \cite{MN} for the fused K-matrices. For more recent results, see e.g.\ \cite{BLN,LBG,RSV2}. 
	
	\section{The Rosso embedding and the Freidel-Maillet type equation}
	\noindent We first make a few conventions and notations. 
	
	\medskip
	\noindent Recall the integers $\mZ=\{0,\pm 1,\pm 2,\ldots\}$ and the natural numbers $\mN=\{0,1,2,\ldots\}$. Let $\mF$ denote a quadratically closed field of characteristic zero. All algebras in this paper are associative, over $\mF$, and have a multiplicative identity. Let $q$ denote a nonzero scalar in $\mF$ that is not a root of unity. For $n \in \mZ$, define 
	\begin{equation*}
		[n]_q=\frac{q^n-q^{-n}}{q-q^{-1}}. 
	\end{equation*}
	
	\noindent We also define the short-hand notation 
	\begin{equation*}
		c(t)=t-t^{-1}. 
	\end{equation*}
	
	\noindent The algebra $U_q^+$ has a presentation with two generators $A,B$ and the \textit{$q$-Serre relations} 
	\begin{equation*}\label{eq:qSerre1}
		A^3B-[3]_qA^2BA+[3]_qABA^2-BA^3=0, 
	\end{equation*}
	\begin{equation*}\label{eq:qSerre2}
		B^3A-[3]_qB^2AB+[3]_qBAB^2-AB^3=0. 
	\end{equation*}
	
	\noindent Our main result is motivated by a presentation for $U_q^+$ of \textit{Freidel-Maillet type}, due to Baseilhac \cite[Theorem 2.10]{baseilhac1}. The main equation of the presentation involves an R-matrix, an $\widehat{\text{R}}$-matrix with scalar entries, and a K-matrix. 
	
	\medskip
	\noindent
	Our main result is also related to \cite{LBG}, where fused K-operators for the alternating central extension of the $q$-Onsager algebra \cite{ter_qOns2} are constructed by a fusion technique similar to ours. There are two differences. First, the K-operators of \cite{LBG} are valued in the alternating central extension of the $q$-Onsager algebra, whereas ours are valued in $U_q^+$ itself. Second, our K-matrices are given in closed form for all spins (see Definition \ref{def:K} and Remark \ref{rm:Dfactor} below); providing this closed form is the goal of the present paper, while in \cite{LBG} the fused K-operators are expressed recursively. 
	
	\medskip
	\noindent Let $t$ denote an indeterminate. We adopt the R-matrix from \cite{baseilhac1} and define the $\widehat{\text{R}}$-matrix, as follows: 
	\begin{equation}\label{eq:R^1/2}
		R^{(\frac{1}{2},\frac{1}{2})}(t)=
		\begin{pmatrix}
			c(qt) & 0 & 0 & 0 \\
			0 & c(t) & c(q) & 0 \\
			0 & c(q) & c(t) & 0 \\
			0 & 0 & 0 & c(qt)
		\end{pmatrix}, 
	\end{equation}
	\begin{equation}\label{eq:hR^1/2}
		\hR^{(\frac{1}{2},\frac{1}{2})}=\diag(q^{\frac{1}{2}},q^{-\frac{1}{2}},q^{-\frac{1}{2}},q^{\frac{1}{2}})=q^{2\diag(\frac{1}{2},-\frac{1}{2}) \otimes \diag(\frac{1}{2},-\frac{1}{2})}, 
	\end{equation}
	where $\diag(~)$ denotes the diagonal matrix with the given diagonal. 
	
	\medskip
	\noindent We display the K-matrix from \cite{baseilhac1} using the following language: an embedding of $U_q^+$ into a $q$-shuffle algebra $\mV$ due to Rosso \cite{rosso1,rosso2}. 
	
	\medskip
	\noindent We first recall the $q$-shuffle algebra $\mV$. Let $x,y$ denote noncommuting indeterminates. We call $x$ and $y$ \textit{letters}. Let $\mV$ denote the free algebra generated by the letters $x,y$. For $n \in \mN$, the product of $n$ letters is called a \textit{word} of \textit{length} $n$. The word of length $0$ is called \textit{trivial} and denoted by $\mone$. The vector space $\mV$ has a basis consisting of all words; this basis is called \textit{standard}. 
	
	\medskip
	\noindent We now equip $\mV$ with another algebra structure, called the \textit{$q$-shuffle algebra} \cite{rosso1,rosso2}. The $q$-shuffle product is denoted by $\star$. We adopt the description by Green \cite{green}. 
	
	\begin{itemize}
		\item For $v \in \mV$, 
		\begin{equation*}\label{star1}
			\mone \star v=v \star \mone=v. 
		\end{equation*}
		\item For the letters $u,v$, 
		\begin{equation*}\label{star2}
			u \star v=uv+vuq^{\langle u,v \rangle}, 
		\end{equation*}	
		where 
		\begin{equation*}
			\langle x,x \rangle=\langle y,y \rangle =2, \hspace{4em}\langle x,y \rangle=\langle y,x \rangle=-2.
		\end{equation*}
		\item For a letter $u$ and a word $v=v_1v_2 \cdots v_n$ in $\mV$ with $n \geq 2$, 
		\begin{equation*}\label{star3.1}
			u \star v=\sum_{i=0}^n v_1 \cdots v_iuv_{i+1} \cdots v_n q^{\langle u,v_1 \rangle+\cdots+\langle u,v_i \rangle}, 
		\end{equation*}
		\begin{equation*}\label{star3.2}
			v \star u=\sum_{i=0}^n v_1 \cdots v_iuv_{i+1} \cdots v_n q^{\langle u,v_n \rangle+\cdots+\langle u,v_{i+1} \rangle}. 
		\end{equation*}
		\item For words $u=u_1u_2 \cdots u_r$ and $v=v_1v_2 \cdots v_s$ in $\mV$ with $r,s \geq 2$, 
		\begin{equation*}\label{star4.1}
			u \star v=u_1((u_2 \cdots u_r) \star v)+v_1(u \star (v_2 \cdots v_s))q^{\langle v_1,u_1 \rangle+\cdots+\langle v_1,u_r \rangle}, 
		\end{equation*}
		\begin{equation*}\label{star4.2}
			u \star v=(u \star (v_1 \cdots v_{s-1}))v_s+((u_1 \cdots u_{r-1}) \star v)u_rq^{\langle u_r,v_1 \rangle+\cdots+\langle u_r,v_s \rangle}. 
		\end{equation*}
	\end{itemize}
	
	\noindent By \cite{rosso1,rosso2}, the vector space $\mV$, equipped with the $q$-shuffle product $\star$, becomes an algebra. Moreover, $x,y$ satisfy 
	\begin{equation*}
		x \star x \star x \star y-[3]_qx \star x \star y \star x+[3]_qx \star y \star x \star x-y \star x \star x \star x=0, 
	\end{equation*}
	\begin{equation*}
		y \star y \star y \star x-[3]_qy \star y \star x \star y+[3]_qy \star x \star y \star y-x \star y \star y \star y=0. 
	\end{equation*}
	
	\noindent As a result, there exists an algebra homomorphism $\natural$ from $U_q^+$ to the $q$-shuffle algebra $\mV$ that sends $A \mapsto x$ and $B \mapsto y$. By \cite[Theorem 15]{rosso2} the map $\natural$ is injective. Let $U$ denote the subalgebra of $\mV$ generated by $x,y$ with respect to the $q$-shuffle product. The map $\natural:U_q^+ \to U$ is an algebra isomorphism. Throughout this paper, we identify $U_q^+$ with $U$ via $\natural$. 
	
	\medskip
	\noindent We are about to display the K-matrix from \cite{baseilhac1} in closed form. Entries of the K-matrix are generating functions of the following type of words in $U$. 
	
	\begin{definition}\label{def:alt}\rm
		(See \cite[Definition 5.2]{ter_alternating}.) We define the following words in $U$. 
		\begin{align*}
			& W_0=x, && W_{-1}=xyx, && W_{-2}=xyxyx, && W_{-3}=xyxyxyx && \ldots\\
			& W_1=y, && W_2=yxy, && W_3=yxyxy, && W_4=yxyxyxy && \ldots\\
			& \tG_1=xy, && \tG_2=xyxy, && \tG_3=xyxyxy, && \tG_4=xyxyxyxy && \ldots\\
			& G_1=yx, && G_2=yxyx, && G_3=yxyxyx, && G_4=yxyxyxyx && \ldots
		\end{align*}
		
		\noindent These words are said to be \textit{alternating}. 
	\end{definition}
	
	\noindent In \cite[Section 10]{ter_alternating} it is showed that each of the following form a PBW basis for $U$: 
	\begin{itemize}
		\item $\{W_{-n}\}_{n=1}^\infty$, $\{W_{n+1}\}$, $\{\tG_n\}_{n=1}^\infty$;
		\item $\{W_{-n}\}_{n=1}^\infty$, $\{W_{n+1}\}$, $\{G_n\}_{n=1}^\infty$. 
	\end{itemize}
	
	\noindent In addition, the alternating words are connected to the doubly alternating words; see \cite[Section 5]{dalt}. 
	
	\medskip
	\noindent For notational convenience, we let $\tG_0=G_0=\mone$. 
	
	\begin{definition}\label{def:altgen}\rm
		(See \cite[Definition 9.1]{ter_alternating}.) We define the following generating functions. 
		\begin{align*}
			W^-(t)&=\sum_{n \in \mN}W_{-n}t^n, & W^+(t)&=\sum_{n \in \mN}W_{n+1}t^n, \\
			\tG(t)&=\sum_{n \in \mN}\tG_nt^n, & G(t)&=\sum_{n \in \mN}G_nt^n. 
		\end{align*}
	\end{definition}
	
	\noindent We are now ready to display the K-matrix from \cite{baseilhac1} in closed form. 
	
	\begin{definition}\label{def:K1/2}\rm
		We define the $2 \times 2$ matrix 
		\begin{equation*}\label{eq:K1/2}
			K^{(\frac{1}{2})}(t)=
			\begin{pmatrix}
				qtW^-(t^2) & G(t^2) \\
				\tG(t^2) & qtW^+(t^2)
			\end{pmatrix}. 
		\end{equation*}
	\end{definition}
	
	\begin{remark}\label{rm:baseilhacnotation}\rm
		The above matrix can be obtained from the matrix $K(u)$ in \cite[Theorem 2.10]{baseilhac1} up to a scalar multiple via the correspondence \footnote{There is a typo in the equations right above Remark 2.8 of \cite{baseilhac1}: the $G_{k+1}$ and $\tG_{k+1}$ should be exchanged. }
		\begin{align*}
			U &\mapsto t^{-2}, &&\\
			y_{n+1}^+ &\mapsto W_{-n}, & y_{-n}^+ &\mapsto W_{n+1},\\
			\widetilde{z}_{n+1}^+ &\mapsto q^{-1}(q^2-q^{-2})G_{n+1}, & z_{n+1}^+ &\mapsto q^{-1}(q^2-q^{-2})\tG_{n+1},\\
			\bar{k}_+ &\mapsto q^{-\frac{1}{2}}(q+q^{-1})^{-\frac{1}{2}}(q-q^{-1}), & \bar{k}_- &\mapsto q^{-\frac{1}{2}}(q+q^{-1})^{-\frac{1}{2}}(q-q^{-1}).
		\end{align*}
		
		\noindent Here both of $\bar{k}_+,\bar{k}_-$ are mapped to scalars. However in \cite{baseilhac1} it is assumed that $\bar{k}_+\bar{k}_-$ is equal to a fixed nonzero scalar; see \cite[(2.20), (2.31)]{baseilhac1}. This means that one of $\bar{k}_+,\bar{k}_-$ is a free nonzero scalar. We will recover one free nonzero scalar in Appendix C. 
	\end{remark}
	
	\noindent By \cite[Propositions 5.7, 5.10, 5.11]{ter_alternating}, the matrices $R^{(\frac{1}{2},\frac{1}{2})}(t), \hR^{(\frac{1}{2},\frac{1}{2})}, K^{(\frac{1}{2})}(t)$ satisfy the following Freidel-Maillet type equation 
	\begin{equation}\label{eq:FM1/2}
		R^{(\frac{1}{2},\frac{1}{2})}(t/s) \star K^{(\frac{1}{2})}_{1}(s) \star \hR^{(\frac{1}{2},\frac{1}{2})} \star K^{(\frac{1}{2})}_{2}(t)=K^{(\frac{1}{2})}_{2}(t) \star \hR^{(\frac{1}{2},\frac{1}{2})} \star K^{(\frac{1}{2})}_{1}(s) \star R^{(\frac{1}{2},\frac{1}{2})}(t/s). 
	\end{equation}
	
	\noindent This result coincides with \cite[(2.33)]{baseilhac1} under the correspondence from Remark \ref{rm:baseilhacnotation}. 
	
	\medskip
	\noindent In this paper, we will obtain a more general result involving fused R-, $\widehat{\text{R}}$-, and K-matrices of arbitrary meaningful dimensions. This result will be presented in the next section. 
	
	\section{The general Freidel-Maillet type equation}
	\noindent In this section we state our main result, which is the general Freidel-Maillet type equation. In order to do this, we need to define the fused R-, $\widehat{\text{R}}$-, and K-matrices of higher dimensions. For simplicity, from now on we omit the word `fused' unless ambiguity is present. 
	
	\medskip
	\noindent We adopt the recursive definition from \cite{LBG} for the R-matrix. 
	
	\begin{definition}\label{def:E}\rm
		(See \cite[(3.17)]{LBG}.) For $j  \in \frac{1}{2}\mN^+$, define the $(4j+2) \times (2j+2)$ matrix $\cE^{(j+\frac{1}{2})}$ where all the nonzero entries are given as follows: 
		\begin{equation*}
			\cE^{(j+\frac{1}{2})}_{(a,a)}=\left(\frac{[2j+2-a]_q}{[2j+1]_q}\right)^{\frac{1}{2}}, \hspace{2em} \cE^{(j+\frac{1}{2})}_{(a+2j+1,a+1)}=\left(\frac{[a]_q}{[2j+1]_q}\right)^{\frac{1}{2}}, 
		\end{equation*}
		where $1 \leq a \leq 2j+1$. 
	\end{definition}
	
	\begin{definition}\label{def:F}\rm
		(See \cite[(3.23), (3.24)]{LBG}.) For $j  \in \frac{1}{2}\mN^+$, define the $(2j+2) \times (4j+2)$ matrix $\cF^{(j+\frac{1}{2})}$ where all the nonzero entries are given as follows: 
		\begin{equation*}
			\cF^{(j+\frac{1}{2})}_{(a,a)}=\frac{([2j+2-a]_q[2j+1]_q)^\frac{1}{2}}{[2j+2-a]_q+[a-1]_q}, \hspace{2em} \cF^{(j+\frac{1}{2})}_{(a+1,a+2j+1)}=\frac{([a]_q[2j+1]_q])^\frac{1}{2}}{[2j+1-a]_q+[a]_q}, 
		\end{equation*}
		where $1 \leq a \leq 2j+1$. 
	\end{definition}
	
	\begin{definition}\label{def:R}\rm
		(See \cite[(4.10), (4.32), (4.33)]{LBG}.) For $j_1,j_2 \in \frac{1}{2}\mN^+$, define the $(2j_1+1)(2j_2+1) \times (2j_1+1)(2j_2+1)$ matrix $R^{(j_1,j_2)}(t)$ recursively by 
		\begin{equation}\label{eq:Rdef1}
			R^{(\frac{1}{2},j_2+\frac{1}{2})}(t)=\cF^{(j_2+\frac{1}{2})}_{23} R^{(\frac{1}{2},j_2)}_{13}(q^{-\frac{1}{2}}t) R^{(\frac{1}{2},\frac{1}{2})}_{12}(q^{j_2}t) \cE^{(j_2+\frac{1}{2})}_{23}, 
		\end{equation}
		\begin{equation}\label{eq:Rdef2}
			R^{(j_1+\frac{1}{2},j_2)}(t)=\cF^{(j_1+\frac{1}{2})}_{12} R^{(\frac{1}{2},j_2)}_{13}(q^{-j_1}t) R^{(j_1,j_2)}_{23}(q^{\frac{1}{2}}t) \cE^{(j_1+\frac{1}{2})}_{12}, 
		\end{equation}
		where $R^{(\frac{1}{2},\frac{1}{2})}(t)$ is given in \eqref{eq:R^1/2}. 
	\end{definition}
	
	\noindent The $\widehat{\text{R}}$-matrix is defined as follows. 
	
	\begin{definition}\label{def:hR}\rm
		For $j_1,j_2 \in \frac{1}{2}\mN^+$, define the $(2j_1+1)(2j_2+1) \times (2j_1+1)(2j_2+1)$ diagonal matrix $\hR^{(j_1,j_2)}$ by 
		\begin{equation*}
			\hR^{(j_1,j_2)}=q^{2\diag(j_1,j_1-1,\ldots,-j_1) \otimes \diag(j_2,j_2-1,\ldots,-j_2)}. 
		\end{equation*}
	\end{definition}
	
	\noindent Clearly Definitions \ref{def:R}, \ref{def:hR} are compatible with \eqref{eq:R^1/2}, \eqref{eq:hR^1/2}. 
	
	\noindent As we will see in Section 5, the $\widehat{\text{R}}$-matrix is related to the R-matrix and holds many properties similar to those of the R-matrix. 
	
	\medskip
	\noindent Next we will define the K-matrix. The definition depends on a certain type of word in $U$, said to be Catalan. 
	
	\begin{definition}\label{def:Cat}\rm
		(See \cite[Definition 1.3]{ter_catalan}.) Let $\overline{x}=1$ and $\overline{y}=-1$. A word $a_1a_2 \cdots a_n$ is \textit{Catalan} whenever $\overline{a}_1+\overline{a}_2+\cdots+\overline{a}_i \geq 0$ for $1 \leq i \leq n-1$ and $\overline{a}_1+\overline{a}_2+\cdots+\overline{a}_n=0$. The length of a Catalan word is even. For $n \in \mN$, we denote the collection of Catalan words of length $2n$ by $\Cat_n$. 
	\end{definition}
	
	\begin{example}\label{ex:Catn}\rm
		We list the Catalan words of length $\leq 6$. 
		\begin{equation*}
			\mone, \hspace{4em}xy, \hspace{4em}xyxy, \hspace{1em}xxyy, 
		\end{equation*}
		\begin{equation*}
			xyxyxy, \hspace{1em}xxyyxy, \hspace{1em}xyxxyy, \hspace{1em}xxyxyy, \hspace{1em}xxxyyy. 
		\end{equation*}
	\end{example}
	
	\noindent For notational convenience, for $n \in \mN$ define 
	\begin{equation*}
		[n]_q^!=[n]_q[n-1]_q \cdots [1]_q. 
	\end{equation*}
	
	\noindent By convention, $[0]_q^!=1$. 
	
	\begin{definition}\label{def:D^m_n}\rm
		(See \cite[Definitions 4.1, 4.6, 10.5]{uniform}.) For $m \in \mZ$ and $n \in \mN$, define 
		\begin{equation*}
			\Delta^{(m)}_n=\sum_{a_1 \cdots a_{2n} \in \Cat_n}\prod_{i=1}^{2n}[\overline{a}_1+\overline{a}_2+\cdots+\overline{a}_{i-1}+m(\overline{a}_i+1)/2]_q ~ a_1 \cdots a_{2n}. 
		\end{equation*}
		We remark that $\Delta^{(m)}_0=\mone$. 
		
		\begin{proposition}\label{prop:DinU}\rm
			For $m \in \mZ$ and $n \in \mN$, the element $\Delta^{(m)}_n$ lies in $U$. 
		\end{proposition}
		\begin{proof}
			By \cite[Theorem 2.25 and Lemma 4.10]{uniform}, each $\Delta^{(m)}_n$ is a polynomial in elements of $U$ and thus lies in $U$. 
		\end{proof}
		
		\noindent For $m \in \mZ$, define the generating function 
		\begin{equation*}
			\Delta^{(m)}(t)=\sum_{n \in \mN}\Delta^{(m)}_nt^n. 
		\end{equation*}
	\end{definition}
	
	\noindent The generating function $\Delta^{(m)}(t)$ plays a key role in a uniform approach to the three PBW bases for $U_q^+$ due to Damiani, Beck, and Terwilliger; see \cite{uniform}. We remark that the PBW basis due to Terwilliger consists of alternating words, as mentioned under Definition \ref{def:alt}. 
	
	\medskip
	\noindent For notational convenience, we make the following definition. 
	
	\begin{definition}\label{def:-1}\rm
		(See \cite[Lemma 4.3]{PT}.) For $n \geq 1$ and a word $w=a_1a_2 \cdots a_n$, define 
		\begin{equation*}
			x^{-1}w=
			\begin{cases}
				a_2a_3 \cdots a_n,&\hspace{1em}\text{ if }a_1=x;\\
				0,&\hspace{1em}\text{ if }a_1=y. 
			\end{cases}
		\end{equation*}
		
		\noindent By convention, $x^{-1}\mone=0$. 
		
		\medskip 
		\noindent We also define $y^{-1}w$, $wx^{-1}$, $wy^{-1}$ in a similar way. 
		
		\medskip
		\noindent We extend the above definitions linearly to all of $\mV$ and to generating functions on $\mV$. 
	\end{definition}
	
	\noindent Now we are ready to define the K-matrix in closed form. 
	
	\begin{definition}\label{def:K}\rm
		For $j \in \frac{1}{2}\mN^+$, define the $(2j+1) \times (2j+1)$ matrix $K^{(j)}(t)$ with the $(a,b)$-entry given by  
		\begin{equation}\label{eq:K}
			K^{(j)}_{(a,b)}(t)=\varphi(a,b,j)t^{a-b-2j}x^{1-b}\Delta^{(-2j)}(-t^2)y^{a-2j-1}, 
		\end{equation}
		where 
		\begin{equation}\label{eq:f}
			\varphi(a,b,j)=q^{\rho(a,b,j)}\left([2j]_q^!\right)^{-1}\left(\frac{[a-1]_q^![2j+1-b]_q^!}{[b-1]_q^![2j+1-a]_q^!}\right)^\frac{1}{2}
		\end{equation}
		and 
		\begin{equation}\label{eq:h}
			\rho(a,b,j)=(a^2+b^2+6j^2+4ab-6aj-6bj-6a-6b+13j+6)/2.  
		\end{equation}
	\end{definition}
	
	\begin{remark}\label{rm:Dfactor}\rm
		Each entry of each K-matrix $K^{(j)}(t)$ is a generating function with indeterminate $t$, whose coefficients are linear combinations of words in the letters $x,y$; each coefficient lies in $U$ by Proposition \ref{prop:DinU} and \cite[Proposition 9.3]{PT}. 
		
		\medskip
		\noindent Moreover, each coefficient may be written in terms of the alternating PBW basis from \cite{ter_alternating} as follows. By \cite[Theorem 2.25(i)]{uniform}, for $j \in \frac{1}{2}\mN^+$ we have the factorization 
		\begin{equation*}\label{eq:Dfactor}
			\Delta^{(-2j)}(-t^2)=\tG(q^{2j-1}t^2) \star \tG(q^{2j-3}t^2) \star \cdots \star \tG(q^{-2j+1}t^2).
		\end{equation*}
		Note that $x^{-1}$ and $y^{-1}$ sends an alternating word to either an alternating word or $\mone$; see Definition \ref{def:alt}. Using this fact we may express each coefficient in terms of the alternating words. 
		
		\medskip
		\noindent For instance, since $\Delta^{(-2)}(-t^2)=\tG(qt^2) \star \tG(q^{-1}t^2)$, Definition \ref{def:K} yields
		\begin{equation*}
			K^{(1)}_{(a,b)}(t)=\varphi(a,b,1)\, t^{a-b-2}\, x^{1-b}\big(\tG(qt^2) \star \tG(q^{-1}t^2)\big) y^{a-3},
			\hspace{2em} 1 \leq a,b \leq 3.
		\end{equation*}
		To find the coefficient of $t^n$ in $K^{(1)}_{(a,b)}(t)$, we first find the coefficient of $t^{n-a+b+2}$ in $\tG(qt^2) \star \tG(q^{-1}t^2)$, then apply $x^{-1},y^{-1}$ accordingly and multiply by the scalar $\varphi(a,b,1)$. The result will be expressed in terms of the alternating words. 
	\end{remark}
	
	\noindent One can routinely verify that Definition \ref{def:K} is compatible with Definition \ref{def:K1/2}. 
	
	\medskip
	\noindent Our main result is the following Freidel-Maillet type equation. 
	
	\begin{theorem}\label{thm:RKRK}\rm
		For $j_1,j_2 \in \frac{1}{2}\mN^+$, 
		\begin{equation}\label{eq:RKRK}
			R^{(j_1,j_2)}(t/s) \star K^{(j_1)}_{1}(s) \star \hR^{(j_1,j_2)} \star K^{(j_2)}_{2}(t)=K^{(j_2)}_{2}(t) \star \hR^{(j_1,j_2)} \star K^{(j_1)}_{1}(s) \star R^{(j_1,j_2)}(t/s). 
		\end{equation}
	\end{theorem}
	
	\noindent We will get another Freidel-Maillet type equation as a corollary. 	
	
	\begin{corollary}\label{cor:KRKR}\rm
		For $j_1,j_2 \in \frac{1}{2}\mN^+$, 
		\begin{equation}\label{eq:KRKR}
			K^{(j_1)}_{1}(s) \star \hR^{(j_1,j_2)} \star K^{(j_2)}_{2}(t) \star R^{(j_1,j_2)}(s/t)=R^{(j_1,j_2)}(s/t) \star K^{(j_2)}_{2}(t) \star \hR^{(j_1,j_2)} \star K^{(j_1)}_{1}(s). 
		\end{equation}
	\end{corollary}
	
	\begin{remark}\label{rm:mainthm}\rm
		We interpret Theorem \ref{thm:RKRK} in the language of the quantum inverse scattering method. For every pair of spins $j_1,j_2$, the K-matrices $K^{(j_1)},K^{(j_2)}$ are the boundary reflection operators of an open spin chain with Yang--Baxter operators $R^{(j_1,j_2)},\hR^{(j_1,j_2)}$; by \cite{sklyanin}, the associated double-row transfer matrices commute for all values of the spectral parameters. Moreover, evaluating $U_q^+$ on any finite-dimensional $U_q(\widehat{\mathfrak{sl}}_2)$-module (which remains irreducible upon restriction to $U_q^+$ by \cite[Theorem 6.2.4]{AV25}) yields explicit $(2j+1) \times (2j+1)$ scalar K-matrices, whose free boundary parameter $k$ is discussed in Appendix C. We expect these scalar K-matrices to relate to the higher-spin K-matrices of the open XXZ chain \cite{DN}; the fusion relation \eqref{eq:Krecur} matches the fusion procedure of \cite{MN,LBG}.
	\end{remark}
	
	\begin{remark}\label{rm:quasiR}\rm
		A Freidel-Maillet type equation of the same form as \eqref{eq:KRKR} may be produced from Lusztig's quasi R-matrix using properties of the universal R-matrix; see Appendix D. Our approach is of independent interest as it gives a closed form for the K-matrix and uses results only from the $q$-shuffle algebra. This yields a family of explicit algebraic relations over $U_q^+$ that generalize \eqref{eq:FM1/2} and may contribute to the study of integrable systems or higher‑spin representations. 
	\end{remark}
	
	\noindent For the rest of this paper, we prove Theorem \ref{thm:RKRK} and Corollary \ref{cor:KRKR}. We will use a fusion technique analogous to that of \cite{LBG}. Our proof strategy is as follows. In Section 4, we recall some known properties of the R-matrix. In Section 5, we show some properties of the $\widehat{\text{R}}$-matrix. In Section 6, we obtain a recurrence relation for the K-matrix that corresponds to a fusion technique. In Section 7, we prove Theorem \ref{thm:RKRK} and Corollary \ref{cor:KRKR} using the results from Sections 4, 5, 6. 
	
	\section{Some known properties of the R-matrix}
	\noindent In this section, we recall some known properties of the R-matrix $R^{(j_1,j_2)}(t)$ that will be used later. 
	
	\begin{definition}\label{def:H}\rm
		(See \cite[(3.44)]{LBG}.) For $j \in \frac{1}{2}\mN^+$, define the $(2j+2) \times (2j+2)$ diagonal matrix $\cH^{(j+\frac{1}{2})}$ with the $(a,a)$-entry given by 
		\begin{equation*}
			\cH^{(j+\frac{1}{2})}_{(a,a)}=c(q)c(q^2) \cdots c(q^{2j})([2j+2-a]_q+[a-1]_q), 
		\end{equation*}
		where $1 \leq a \leq 2j+2$. 
	\end{definition}
	
	\noindent The following four Lemmas \ref{lem:EFH}--\ref{lem:RRR} are already proved in \cite{LBG}. Here we simply restate them without proof. 
	
	\begin{lemma}\label{lem:EFH}\rm
		(See \cite[(3.15), (3.43), (3.45), (3.46), (3,47)]{LBG}.) For $j \in \frac{1}{2}\mN^+$, \\
		\begin{equation}\label{eq:FE}
			\cF^{(j+\frac{1}{2})} \cE^{(j+\frac{1}{2})}=\mI_{2j+2}, 
		\end{equation}
		\begin{equation}\label{eq:EHF}
			R^{(\frac{1}{2},j)}(q^{j+\frac{1}{2}})=\cE^{(j+\frac{1}{2})} \cH^{(j+\frac{1}{2})} \cF^{(j+\frac{1}{2})}, 
		\end{equation}
		\begin{equation}\label{eq:RE}
			R^{(\frac{1}{2},j)}(q^{j+\frac{1}{2}}) \cE^{(j+\frac{1}{2})}=\cE^{(j+\frac{1}{2})} \cH^{(j+\frac{1}{2})}, 
		\end{equation}
		\begin{equation}\label{eq:FR}
			\cF^{(j+\frac{1}{2})} R^{(\frac{1}{2},j)}(q^{j+\frac{1}{2}})=\cH^{(j+\frac{1}{2})} \cF^{(j+\frac{1}{2})}, 
		\end{equation}
		\begin{equation}\label{eq:EFR}
			R^{(\frac{1}{2},j)}(q^{j+\frac{1}{2}})=\cE^{(j+\frac{1}{2})} \cF^{(j+\frac{1}{2})} R^{(\frac{1}{2},j)}(q^{j+\frac{1}{2}}). 
		\end{equation}
	\end{lemma}
	
	\begin{lemma}\label{lem:Rrecur}\rm
		(See \cite[(4.32) and Lemma 5.9]{LBG}.) For $j_1,j_2 \in \frac{1}{2}\mN^+$, we have the following fusion equalities on the R-matrices \\
		\begin{equation}\label{eq:Rrecur1}
			R^{(j_1+\frac{1}{2},j_2)}(t)=\cF^{(j_1+\frac{1}{2})}_{12} R^{(\frac{1}{2},j_2)}_{13}(q^{-j_1}t) R^{(j_1,j_2)}_{23}(q^{\frac{1}{2}}t) \cE^{(j_1+\frac{1}{2})}_{12}, 
		\end{equation}
		\begin{equation}\label{eq:Rrecur2}
			R^{(j_1+\frac{1}{2},j_2)}(t)=\cF^{(j_1+\frac{1}{2})}_{12} R^{(j_1,j_2)}_{23}(q^{-\frac{1}{2}}t) R^{(\frac{1}{2},j_2)}_{13}(q^{j_1}t) \cE^{(j_1+\frac{1}{2})}_{12}, 
		\end{equation}
		\begin{equation}\label{eq:Rrecur3}
			R^{(j_1,j_2+\frac{1}{2})}(t)=\cF^{(j_2+\frac{1}{2})}_{23} R^{(j_1,\frac{1}{2})}_{12}(q^{-j_2}t) R^{(j_1,j_2)}_{13}(q^{\frac{1}{2}}t) \cE^{(j_2+\frac{1}{2})}_{23}, 
		\end{equation}
		\begin{equation}\label{eq:Rrecur4}
			R^{(j_1,j_2+\frac{1}{2})}(t)=\cF^{(j_2+\frac{1}{2})}_{23} R^{(j_1,j_2)}_{13}(q^{-\frac{1}{2}}t) R^{(j_1,\frac{1}{2})}_{12}(q^{j_2}t) \cE^{(j_2+\frac{1}{2})}_{23}. 
		\end{equation}
	\end{lemma}
	
	\begin{lemma}\label{lem:RR}\rm
		(See \cite[(4.46)]{LBG}.) For $j_1,j_2 \in \frac{1}{2}\mN^+$, the matrix $R^{(j_1,j_2)}(t) R^{(j_1,j_2)}(t^{-1})$ is proportional to $\mI_{(2j_1+1)(2j_2+1)}$ by a nonzero scalar in $\mF[t,t^{-1}]$. 
	\end{lemma}
	
	\begin{lemma}\label{lem:RRR}\rm
		(See \cite[(5.25), (5.26)]{LBG}.) For $j_1,j_2,j_3 \in \frac{1}{2}\mN^+$, we have the following Yang-Baxter equations \\
		\begin{equation}\label{eq:RRR1}
			R^{(j_1,j_2)}_{12}(t_1/t_2) R^{(j_1,j_3)}_{13}(t_1/t_3) R^{(j_2,j_3)}_{23}(t_2/t_3)=R^{(j_2,j_3)}_{23}(t_2/t_3) R^{(j_1,j_3)}_{13}(t_1/t_3) R^{(j_1,j_2)}_{12}(t_1/t_2), 
		\end{equation}
		\begin{equation}\label{eq:RRR2}
			R^{(j_1,j_3)}_{13}(t_1/t_3) R^{(j_2,j_3)}_{23}(t_2/t_3) R^{(j_1,j_2)}_{12}(t_2/t_1)=R^{(j_1,j_2)}_{12}(t_2/t_1) R^{(j_2,j_3)}_{23}(t_2/t_3) R^{(j_1,j_3)}_{13}(t_1/t_3), 
		\end{equation}
		\begin{equation}\label{eq:RRR3}
			R^{(j_2,j_3)}_{23}(t_3/t_2) R^{(j_1,j_2)}_{12}(t_1/t_2) R^{(j_1,j_3)}_{13}(t_1/t_3)=R^{(j_1,j_3)}_{13}(t_1/t_3) R^{(j_1,j_2)}_{12}(t_1/t_2) R^{(j_2,j_3)}_{23}(t_3/t_2). 
		\end{equation}
	\end{lemma}
	
	\section{Some properties of the $\widehat{\text{R}}$-matrix}
	\noindent In this section, we show some properties for the $\widehat{\text{R}}$-matrix $\hR^{(j_1,j_2)}$ that will be used later. Many of these properties are analogs of those appearing in Section 4. 
	
	\medskip
	\noindent For the ease of computation, we write the diagonal matrix $\hR^{(j_1,j_2)}$ in block diagonal form where each block is of size $(2j_2+1) \times (2j_2+1)$. For $j_2 \in \frac{1}{2}\mN^+$, define the diagonal matrix  
	\begin{equation*}
		\omega^{(j_2)}=\diag(q^{j_2},q^{j_2-1},\ldots,q^{-j_2}). 
	\end{equation*}
	
	\begin{proposition}\label{prop:hRblock}\rm
		For $j_1,j_2 \in \frac{1}{2}\mN^+$, 
		\begin{equation}\label{eq:hRblock}
			\hR^{(j_1,j_2)}=\diag((\omega^{(j_2)})^{2j_1},(\omega^{(j_2)})^{2j_1-2},\ldots,(\omega^{(j_2)})^{-2j_1}). 
		\end{equation}
	\end{proposition}
	\begin{proof}
		Follows from Definition \ref{def:hR}. 
	\end{proof}
	
	\begin{lemma}\label{lem:hRrecur}\rm
		For $j_1,j_2 \in \frac{1}{2}\mN^+$, \\
		\begin{equation}\label{eq:hRrecur1}
			\hR^{(j_1+\frac{1}{2},j_2)}=\cF^{(j_1+\frac{1}{2})}_{12} \hR^{(\frac{1}{2},j_2)}_{13} \hR^{(j_1,j_2)}_{23} \cE^{(j_1+\frac{1}{2})}_{12}, 
		\end{equation}
		\begin{equation}\label{eq:hRrecur2}
			\hR^{(j_1+\frac{1}{2},j_2)}=\cF^{(j_1+\frac{1}{2})}_{12} \hR^{(j_1,j_2)}_{23} \hR^{(\frac{1}{2},j_2)}_{13} \cE^{(j_1+\frac{1}{2})}_{12}, 
		\end{equation}
		\begin{equation}\label{eq:hRrecur3}
			\hR^{(j_1,j_2+\frac{1}{2})}=\cF^{(j_2+\frac{1}{2})}_{23} \hR^{(j_1,\frac{1}{2})}_{12} \hR^{(j_1,j_2)}_{13} \cE^{(j_2+\frac{1}{2})}_{23}, 
		\end{equation}
		\begin{equation}\label{eq:hRrecur4}
			\hR^{(j_1,j_2+\frac{1}{2})}=\cF^{(j_2+\frac{1}{2})}_{23} \hR^{(j_1,j_2)}_{13} \hR^{(j_1,\frac{1}{2})}_{12} \cE^{(j_2+\frac{1}{2})}_{23}. 
		\end{equation}
	\end{lemma}
	\begin{proof}
		Since diagonal matrices of the same dimension commutes with each other, it suffices to verify \eqref{eq:hRrecur1} and \eqref{eq:hRrecur3}. For simplicity, write $\omega=\omega^{(j_2)}$. 
		
		\medskip
		\noindent We first verify \eqref{eq:hRrecur1}. By \eqref{eq:hRblock}, we have 
		\begin{align*}
			\hR^{(\frac{1}{2},j_2)}_{13} \hR^{(j_1,j_2)}_{23}&=\diag(\omega,\ldots,\omega,\omega^{-1},\ldots,\omega^{-1})\diag(\omega^{2j_1},\ldots,\omega^{-2j_1},\omega^{2j_1},\ldots,\omega^{-2j_1})\\
			&=\diag(\omega^{2j_1+1},\omega^{2j_1-1},\ldots,\omega^{-2j_1+1},\omega^{2j_1-1},\ldots,\omega^{-2j_1+1},\omega^{-2j_1-1}). 
		\end{align*}
		
		\noindent This is a block diagonal matrix where each block is of size $(2j_2+1) \times (2j_2+1)$. Write $\cF^{(j_1+\frac{1}{2})}_{12}$ (resp.\ $\cE^{(j_1+\frac{1}{2})}_{12}$) as a block matrix where each block is of size $(2j_2+1) \times (2j_2+1)$, then the $(a,b)$-block is $\cF^{(j_1+\frac{1}{2})}_{(a,b)}\mI_{2j_2+1}$ (resp.\ $\cE^{(j_1+\frac{1}{2})}_{(a,b)}\mI_{2j_2+1}$). By \eqref{eq:hRblock} we have 
		\begin{align*}
			\cF^{(j_1+\frac{1}{2})}_{12} \hR^{(\frac{1}{2},j_2)}_{13} \hR^{(j_1,j_2)}_{23} \cE^{(j_1+\frac{1}{2})}_{12}
			=\hR^{(j_1+\frac{1}{2},j_2)}. 
		\end{align*}
		
		\noindent We now verify \eqref{eq:hRrecur3}. By \eqref{eq:hRblock}, we have 
		\begin{align*}
			\hR^{(j_1,\frac{1}{2})}_{12} \hR^{(j_1,j_2)}_{13}=\diag(\omega^{2j_1},q^{-2j_1}\omega^{2j_1},q^{-1}\omega^{2j_1-2},q^{1-2j_1}\omega^{2j_1-2},\ldots,q^{-2j_1}\omega^{-2j_1},\omega^{-2j_1}). 
		\end{align*}
		
		\noindent This is a $(2j_1+1) \times (2j_1+1)$ block diagonal matrix where the $(i,i)$-block is equal to 
		\begin{align*}
			\diag(q^{1-i}\omega^{2j_1+2-2i},q^{i-2j_1-1}\omega^{2j_1+2-2i}). 
		\end{align*}
		
		\noindent Note that $\cF^{(j_1+\frac{1}{2})}_{23}$ (resp.\ $\cE^{(j_1+\frac{1}{2})}_{23}$) is a $(2j_1+1) \times (2j_1+1)$ block diagonal matrix where each block is equal to $\cF^{(j_1+\frac{1}{2})}$ (resp.\ $\cE^{(j_1+\frac{1}{2})}$). Writing 
		\begin{align*}
			\cF^{(j_2+\frac{1}{2})}_{23} \hR^{(j_1,\frac{1}{2})}_{12} \hR^{(j_1,j_2)}_{13} \cE^{(j_2+\frac{1}{2})}_{23}
		\end{align*}
		as a $(2j_1+1) \times (2j_1+1)$ block diagonal matrix, the $(i,i)$-block is equal to 
		\begin{align*}
			\cF^{(j_1+\frac{1}{2})}\diag(q^{1-i}\omega^{2j_1+2-2i},q^{i-2j_1-1}\omega^{2j_1+2-2i})\cE^{(j_1+\frac{1}{2})}, 
		\end{align*}
		which is equal to 
		\begin{align*}
			(\omega^{(j_2+\frac{1}{2})})^{2j_1+2-2i}. 
		\end{align*}
		
		\noindent By the above discussion and \eqref{eq:hRblock}, we have verified \eqref{eq:hRrecur3}. 
	\end{proof}
	
	\noindent The following result shows that the $\widehat{\text{R}}$-matrix is a limiting case of the R-matrix. 
	
	\begin{proposition}\label{prop:lim}\rm
		For $j_1,j_2 \in \frac{1}{2}\mN^+$, 
		
		\begin{equation}\label{eq:lim}
			\lim_{t \to \infty}\frac{R^{(j_1,j_2)}(t)}{t^{4j_1j_2}}=q^{2j_1j_2}\hR^{(j_1,j_2)}. 
		\end{equation}
	\end{proposition}
	\begin{proof}
		We use induction on $k=j_1+j_2$. 
		
		\medskip
		\noindent By \ref{eq:R^1/2}, we have that \eqref{eq:lim} holds for the pair $(\frac{1}{2},\frac{1}{2})$, so it holds for $k=1$. 
		
		\medskip
		\noindent Now assume \eqref{eq:lim} holds for any pair $(j_1,j_2)$ with $j_1+j_2 \leq k$. We will show that \eqref{eq:lim} holds for any pair $(j_1,j_2)$ with $j_1+j_2=k+\frac{1}{2}$. 
		
		\medskip
		\noindent Without loss of generality we assume $j_1>\frac{1}{2}$. By \eqref{eq:Rrecur1}, 
		\begin{align*}
			\lim_{t \to \infty}\frac{R^{(j_1,j_2)}(t)}{t^{4j_1j_2}}&=\lim_{t \to \infty}t^{-4j_1j_2}\cF^{(j_1)}_{12} R^{(\frac{1}{2},j_2)}_{13}(q^{\frac{1}{2}-j_1}t) R^{(j_1-\frac{1}{2},j_2)}_{23}(q^{\frac{1}{2}}t) \cE^{(j_1)}_{12}\\
			&=\cF^{(j_1)}_{12} \lim_{t \to \infty}\left(\frac{R^{(\frac{1}{2},j_2)}_{13}(q^{\frac{1}{2}-j_1}t)}{(q^{\frac{1}{2}-j_1}t)^{2j_2}} \frac{R^{(j_1-\frac{1}{2},j_2)}_{23}(q^{\frac{1}{2}}t)}{(q^{\frac{1}{2}}t)^{4(j_1-\frac{1}{2})j_2}}\right) \cE^{(j_1)}_{12}. 
		\end{align*}
		
		\noindent By the inductive hypothesis, we have 
		\begin{align*}
			\lim_{t \to \infty}\frac{R^{(\frac{1}{2},j_2)}(q^{\frac{1}{2}-j_1}t)}{(q^{\frac{1}{2}-j_1}t)^{2j_2}}=q^{j_2}\hR^{(\frac{1}{2},j_2)}, 
		\end{align*}
		\begin{align*}
			\lim_{t \to \infty}\frac{R^{(j_1-\frac{1}{2},j_2)}(q^{\frac{1}{2}}t)}{(q^{\frac{1}{2}}t)^{4(j_1-\frac{1}{2})j_2}}=q^{2(j_1-\frac{1}{2})j_2}\hR^{(j_1-\frac{1}{2},j_2)}. 
		\end{align*}
		
		\noindent By the above discussion and \eqref{eq:hRrecur1}, we have 
		\begin{align*}
			\lim_{t \to \infty}\frac{R^{(j_1,j_2)}(t)}{t^{4j_1j_2}}&=q^{2j_1j_2}\cF^{(j_1)}_{12} \hR^{(\frac{1}{2},j_2)}_{13} \hR^{(j_1-\frac{1}{2},j_2)}_{23} \cE^{(j_1)}_{12}= q^{2j_1j_2}\hR^{(j_1,j_2)}. 
		\end{align*}
		
		\noindent By induction, we have shown that \eqref{eq:lim} holds for all $j_1,j_2 \in \frac{1}{2}\mN^+$. 
	\end{proof}
	
	\noindent An immediate consequence is the following equations on the R- and $\widehat{\text{R}}$-matrices that resemble the Yang-Baxter equations. 
	
	\begin{lemma}\label{lem:hRhRR}\rm
		For $j_1,j_2,j_3 \in \frac{1}{2}\mN^+$, \\
		\begin{equation}\label{eq:hRhRR4}
			R^{(j_1,j_2)}_{12}(t) \hR^{(j_1,j_3)}_{13} \hR^{(j_2,j_3)}_{23}=\hR^{(j_2,j_3)}_{23} \hR^{(j_1,j_3)}_{13} R^{(j_1,j_2)}_{12}(t), 
		\end{equation}
		\begin{equation}\label{eq:hRhRR6}
			R^{(j_1,j_3)}_{13}(t) \hR^{(j_2,j_3)}_{23} \hR^{(j_1,j_2)}_{12}=\hR^{(j_1,j_2)}_{12} \hR^{(j_2,j_3)}_{23} R^{(j_1,j_3)}_{13}(t), 
		\end{equation}
		\begin{equation}\label{eq:hRhRR2}
			R^{(j_2,j_3)}_{23}(t) \hR^{(j_1,j_2)}_{12} \hR^{(j_1,j_3)}_{13}=\hR^{(j_1,j_3)}_{13} \hR^{(j_1,j_2)}_{12} R^{(j_2,j_3)}_{23}(t), 
		\end{equation}
		\begin{equation}\label{eq:hRhRR1}
			R^{(j_1,j_2)}_{12}(t) \hR^{(j_2,j_3)}_{23} \hR^{(j_1,j_3)}_{13}=\hR^{(j_1,j_3)}_{13} \hR^{(j_2,j_3)}_{23} R^{(j_1,j_2)}_{12}(t), 
		\end{equation}
		\begin{equation}\label{eq:hRhRR3}
			R^{(j_1,j_3)}_{13}(t) \hR^{(j_1,j_2)}_{12} \hR^{(j_2,j_3)}_{23}=\hR^{(j_2,j_3)}_{23} \hR^{(j_1,j_2)}_{12} R^{(j_1,j_3)}_{13}(t), 
		\end{equation}
		\begin{equation}\label{eq:hRhRR5}
			R^{(j_2,j_3)}_{23}(t) \hR^{(j_1,j_3)}_{13} \hR^{(j_1,j_2)}_{12}=\hR^{(j_1,j_2)}_{12} \hR^{(j_1,j_3)}_{13} R^{(j_2,j_3)}_{23}(t). 
		\end{equation}
	\end{lemma}
	\begin{proof}
		We first verify \eqref{eq:hRhRR4}. In \eqref{eq:RRR1} set $t_1=t\hspace{0.1em}t_2,t_3=1$ and divide both side by $(tt_2)^{4j_1j_3}t_2^{4j_2j_3}$. Then 
		\begin{align*}
			R^{(j_1,j_2)}_{12}(t) \frac{R^{(j_1,j_3)}_{13}(tt_2)}{(tt_2)^{4j_1j_3}} \frac{R^{(j_2,j_3)}_{23}(t_2)}{t_2^{4j_2j_3}}=\frac{R^{(j_2,j_3)}_{23}(t_2)}{t_2^{4j_2j_3}} \frac{R^{(j_1,j_3)}_{13}(tt_2)}{(tt_2)^{4j_1j_3}} R^{(j_1,j_2)}_{12}(t). 
		\end{align*}
		
		\noindent Now let $t_2 \to \infty$ and simplify the result using \eqref{eq:lim}. We obtain \eqref{eq:hRhRR4}. 
		
		\medskip
		\noindent The remaining identities can be verified in a similar way. 
	\end{proof}
	
	\section{A recurrence relation for the K-matrix}
	\noindent In this section, we obtain a recurrence relation for the K-matrix $K^{(j)}(t)$ that will be used later. In order to do this, we first give an alternative closed form for $K^{(j)}(t)$. 
	
	\medskip
	\noindent The following definition will be useful. 
	
	\begin{definition}\label{def:zeta}\rm
		(See \cite[Page 5]{ter_catalan}.) Let $\zeta:\mV \to \mV$ denote the unique $\mF$-linear map given by 
		\begin{itemize}
			\item $\zeta(x)=y$ and $\zeta(y)=x$; 
			\item for a word $a_1 \cdots a_n$, 
			\begin{equation*}
				\zeta(a_1 \cdots a_n)=\zeta(a_n) \cdots \zeta(a_1). 
			\end{equation*}
		\end{itemize}
	\end{definition}
	
	\noindent Clearly the map $\zeta$ is an antiautomorphism on the free algebra $\mV$. Moreover, one can routinely check that the map $\zeta$ is an antiautomorphism on the $q$-shuffle algebra $\mV$. In other words, for $v,w \in \mV$ we have 
	\begin{equation*}
		\zeta(vw)=\zeta(w)\zeta(v), \hspace{4em} \zeta(v \star w)=\zeta(w) \star \zeta(v). 
	\end{equation*}
	
	\begin{example}\rm
		We have 
		\begin{align*}
			&\zeta(W^-(t))=W^+(t), && \zeta(W^+(t))=W^-(t), \\
			&\zeta(\tG(t))=G(t), && \zeta(G(t))=\tG(t). 
		\end{align*}
	\end{example}
	
	\noindent Motivate by the above example, we make the following definitions. 
	
	\begin{definition}\label{def:tD^m_n}\rm
		For $m \in \mZ$ and $n \in \mN$, define 
		\begin{equation*}
			\tDelta^{(m)}_n=\zeta(\Delta^{(m)}_n). 
		\end{equation*}
	\end{definition}
	
	\begin{definition}\label{def:tD(t)}\rm
		For $m \in \mZ$, define the generating function 
		\begin{equation*}
			\tDelta^{(m)}(t)=\zeta(\Delta^{(m)}(t))=\sum_{n \in \mN}\tDelta^{(m)}_nt^n. 
		\end{equation*}
	\end{definition}
	
	\noindent The following result gives a symmetry between the generating functions $\tDelta^{(m)}(t)$ and $\Delta^{(m)}(t)$ for $m \in \mZ$. 
	
	\begin{lemma}\label{lem:D&tD}\rm
		For $m,l,r \in \mN$ with $l,r \leq m$, 
		\begin{equation}\label{eq:D&tD}
			[m-l]_q^![m-r]_q^!x^{-l}\Delta^{(-m)}(-t)y^{-r}=[l]_q^![r]_q^!t^{l+r-m}y^{l-m}\tDelta^{(-m)}(-t)x^{r-m}. 
		\end{equation}
	\end{lemma}
	\begin{proof}
		For $n \in \mN$, we compare the coefficients of $t^n$ on both sides of \eqref{eq:D&tD}. 
		
		\medskip
		\noindent If $l>n$ or $r>n$, by Definition \ref{def:D^m_n} the coefficients of $t^n$ on both sides of \eqref{eq:D&tD} are both $0$. 
		
		\medskip
		\noindent Now assume that $l,r \leq n$. It suffices to show that 
		\begin{equation}\label{eq:D_n}
			(-1)^n[m-l]_q^![m-r]_q^!x^{-l}\Delta^{(-m)}_ny^{-r}
		\end{equation}
		is equal to 
		\begin{equation}\label{eq:tD_n}
			(-1)^{n+m-l-r}[l]_q^![r]_q^!t^{l+r-m}y^{l-m}\tDelta^{(-m)}_{n+m-l-r}x^{r-m}. 
		\end{equation}
		
		\noindent We write \eqref{eq:D_n} and \eqref{eq:tD_n} as linear combinations of words. We first show that the set of words with nonzero coefficient in \eqref{eq:D_n} is equal to the set of words with nonzero coefficients in \eqref{eq:tD_n}. In fact, 
		
		\begin{center}
			a word $w$ has nonzero coefficient in \eqref{eq:D_n} \\
			$\Leftrightarrow$ $x^lwy^r$ is a Catalan word of length $2n$ and height $\leq m$\\
			$\Leftrightarrow$ $x^{m-l}\tilde{w}y^{m-r}$ is a Catalan word of length $2(n+m-l-r)$ and height $\leq m$, \\
			\hspace{12em} where $\tilde{w}$ is obtained from $w$ by switching $x,y$\\
			$\Leftrightarrow$ a word $w$ has nonzero coefficient in \eqref{eq:tD_n}. 
		\end{center}
		
		\noindent Now, we only need to show that for a word $w$ with nonzero coefficient in \eqref{eq:D_n}, the coefficient of $w$ in \eqref{eq:D_n} is equal to the coefficient of $w$ in \eqref{eq:tD_n}. This can be routinely verified using Definition \ref{def:D^m_n}. 
	\end{proof}
	
	\noindent We now give an alternative closed form for the K-matrix. 
	
	\begin{proposition}\label{prop:Kalt}\rm
		For $j \in \frac{1}{2}\mN^+$, the $(2j+1) \times (2j+1)$ matrix $K^{(j)}(t)$ is given by 
		\begin{equation}\label{eq:Kalt}
			K^{(j)}_{(a,b)}(t)=\psi(a,b,j)t^{b-a-2j}y^{b-2j-1}\tDelta^{(-2j)}(-t^2)x^{1-a}, 
		\end{equation}
		where  
		\begin{equation}\label{eq:g}
			\psi(a,b,j)=q^{\rho(a,b,j)}\left([2j]_q^!\right)^{-1}\left(\frac{[b-1]_q^![2j+1-a]_q^!}{[a-1]_q^![2j+1-b]_q^!}\right)^\frac{1}{2}
		\end{equation}
		and $\rho(a,b,j)$ is given in \eqref{eq:h}. 
	\end{proposition}
	\begin{proof}
		Follows from Definition \ref{def:K} and Lemma \ref{lem:D&tD}. 
	\end{proof}
	
	\noindent The following lemma will be useful for obtaining a recurrence relation for $K^{(j)}(t)$. 
	
	\begin{lemma}\label{lem:Drecur}\rm
		For $m,l,r \in \mN$ with $l \leq m$ and $1 \leq r \leq m$, 
		\begin{equation}\label{eq:Drecur1}
			\begin{split}
				x^{-l}\Delta^{(-m-1)}(-t^2)y^{-r}&=q^{2l}[m+1]_qt^2W^-(q^mt^2) \star x^{-l}\Delta^{(-m)}(-q^{-1}t^2)y^{1-r}\\
				&+q^{l-1}[l]_q[m+1]_qtG(q^mt^2) \star x^{1-l}\Delta^{(-m)}(-q^{-1}t^2)y^{1-r}, 
			\end{split}
		\end{equation}
		\begin{equation}\label{eq:Drecur2}
			\begin{split}
				y^{-l}\tDelta^{(-m-1)}(-t^2)x^{-r}&=q^{2l}[m+1]_qtW^+(q^mt^2) \star y^{-l}\tDelta^{(-m)}(-q^{-1}t^2)x^{1-r}\\
				&+q^{l-1}[l]_q[m+1]_qt\tG(q^mt^2) \star y^{1-l}\tDelta^{(-m)}(-q^{-1}t^2)x^{1-r}. 
			\end{split}
		\end{equation}
	\end{lemma}
	\begin{proof}
		We first show \eqref{eq:Drecur1}. 
		
		\medskip
		\noindent By \cite[Theorem 2.25(i)]{uniform}, we have 
		\begin{equation}\label{eq:D=}
			\Delta^{(-m-1)}(-t^2)=\tG(q^{m}t^2) \star \tG(q^{m-2}t^2) \star \cdots \star \tG(q^{-m}t^2). 
		\end{equation}
		
		\noindent On both sides of \eqref{eq:D=} apply $y^{-1}$ on the right and simplify the result using \cite[Lemma 9.2]{ter_alternating}. This yields 
		\begin{equation}\label{eq:Dy^-1=}
			\Delta^{(-m-1)}(-t^2)y^{-1}=[m+1]_qt^2W^-(q^mt^2) \star \Delta^{(-m)}(-q^{-1}t^2). 
		\end{equation}
		
		\noindent On both sides of \eqref{eq:Dy^-1=} apply $y^{-1}$ on the right for $1-r$ times. Since all the words appearing in the expression for $W^-(q^mt^2)$ ends with $x$, this yields 
		\begin{equation}\label{eq:Dy^1-r=}
			\Delta^{(-m-1)}(-t^2)y^{-r}=[m+1]_qt^2W^-(q^mt^2) \star \Delta^{(-m)}(-q^{-1}t^2)y^{1-r}. 
		\end{equation}
		
		\noindent Now we show \eqref{eq:Drecur1} by induction on $l$. 
		
		\medskip
		\noindent The case $l=0$ is exactly \eqref{eq:Dy^1-r=}, which we have just showed. 
		
		\medskip
		\noindent Assume \eqref{eq:Drecur1} holds for an $l \in \mN$, then we have 
		\begin{align*}
			x^{-l-1}\Delta^{(-m-1)}(-t^2)y^{-r}&=x^{-1}\left(q^{2l}[m+1]_qt^2W^-(q^mt^2) \star x^{-l}\Delta^{(-m)}(-q^{-1}t^2)y^{1-r}\right)\\
			&+x^{-1}\left(q^{l-1}[l]_q[m+1]_qt^2G(q^mt^2) \star x^{1-l}\Delta^{(-m)}(-q^{-1}t^2)y^{1-r}\right)\\
			&=q^{2l+2}[m+1]_qt^2W^-(q^mt^2) \star x^{-l-1}\Delta^{(-m)}(-q^{-1}t^2)y^{1-r}\\
			&+q^{2l}[m+1]_qt^2G(q^mt^2) \star x^{-l}\Delta^{(-m)}(-q^{-1}t^2)y^{1-r}\\
			&+q^{l-1}[l]_q[m+1]_qt^2G(q^mt^2) \star x^{-l}\Delta^{(-m)}(-q^{-1}t^2)y^{1-r}\\
			&=q^{2l+2}[m+1]_qt^2W^-(q^mt^2) \star x^{-l-1}\Delta^{(-m)}(-q^{-1}t^2)y^{1-r}\\
			&+q^{l}[l+1]_q[m+1]_qt^2G(q^mt^2) \star x^{-l}\Delta^{(-m)}(-q^{-1}t^2)y^{1-r}. 
		\end{align*}
		
		\noindent We have showed \eqref{eq:Drecur1}. 
		
		\medskip
		\noindent Applying $\zeta$ to both sides of \eqref{eq:Drecur1}, we obtain \eqref{eq:Drecur2}. 
	\end{proof}
	
	\begin{remark}\label{rm:Drecur}\rm
		Using Lemma \ref{lem:Drecur}, we can recursively write entries of the K-matrix $K^{(j)}(t)$ in terms of the generating functions $W^-(t),W^+(t),\tG(t),G(t)$. It is straightforward to verify that, in particular, entries on first (or last) row (or column) can be written in closed form in terms of $W^-(t),W^+(t),\tG(t),G(t)$. 
	\end{remark}
	
	\noindent We are now ready to show the recurrence relation for $K^{(j)}(t)$. We remark that this result is an analog of the fusion technique of the K-matrices in \cite[Definition 5.6]{LBG}. 
	
	\begin{proposition}\label{prop:Krecur}\rm
		For $j \in \frac{1}{2}\mN^+$, 
		\begin{equation}\label{eq:Krecur}
			K^{(j+\frac{1}{2})}(t)=\cF^{(j+\frac{1}{2})} \star K^{(\frac{1}{2})}_{1}(q^{j}t) \star \hR^{(\frac{1}{2},j)} \star K^{(j)}_{2}(q^{-\frac{1}{2}}t) \star \cE^{(j+\frac{1}{2})}. 
		\end{equation}
	\end{proposition}
	\begin{proof}
		Routine computation using Definition \ref{def:K}, Proposition \ref{prop:Kalt}, and Lemma \ref{lem:Drecur}. See Appendix A for the detailed computation. 
	\end{proof}
	
	\section{The Freidel-Maillet type equation}
	\noindent In this section, we show Theorem \ref{thm:RKRK} using induction. The proof technique in this section is along the lines of \cite[Section 5.2]{LBG}; see also \cite[Section 4.3]{RSV2}. We first show two lemmas that contribute to the inductive step. 
	
	\begin{lemma}\label{lem:RKRKind1}\rm
		Given $j_1,j_2 \in \frac{1}{2}\mN^+$. Suppose we have \\
		\begin{equation}\label{eq:RKRKind1.1}
			R^{(\frac{1}{2},j_1)}(t/s) \star K^{(\frac{1}{2})}_{1}(s) \star \hR^{(\frac{1}{2},j_1)} \star K^{(j_1)}_{2}(t)=K^{(j_1)}_{2}(t) \star \hR^{(\frac{1}{2},j_1)} \star K^{(\frac{1}{2})}_{1}(s) \star R^{(\frac{1}{2},j_1)}(t/s), 
		\end{equation}
		\begin{equation}\label{eq:RKRKind1.2}
			R^{(\frac{1}{2},j_2)}(t/s) \star K^{(\frac{1}{2})}_{1}(s) \star \hR^{(\frac{1}{2},j_2)} \star K^{(j_2)}_{2}(t)=K^{(j_2)}_{2}(t) \star \hR^{(\frac{1}{2},j_2)} \star K^{(\frac{1}{2})}_{1}(s) \star R^{(\frac{1}{2},j_2)}(t/s), 
		\end{equation}
		\begin{equation}\label{eq:RKRKind1.3}
			R^{(j_1,j_2)}(t/s) \star K^{(j_1)}_{1}(s) \star \hR^{(j_1,j_2)} \star K^{(j_2)}_{2}(t)=K^{(j_2)}_{2}(t) \star \hR^{(j_1,j_2)} \star K^{(j_1)}_{1}(s) \star R^{(j_1,j_2)}(t/s). 
		\end{equation}
		Then 
		\begin{align}\label{eq:RKRKind1result}
			\begin{split}
				&R^{(j_1+\frac{1}{2},j_2)}(t/s) \star K^{(j_1+\frac{1}{2})}_{1}(s) \star \hR^{(j_1+\frac{1}{2},j_2)} \star K^{(j_2)}_{2}(t)\\
				&\hspace{10em}=K^{(j_2)}_{2}(t) \star \hR^{(j_1+\frac{1}{2},j_2)} \star K^{(j_1+\frac{1}{2})}_{1}(s) \star R^{(j_1+\frac{1}{2},j_2)}(t/s). 
			\end{split}
		\end{align}
	\end{lemma}
	\begin{proof}
		Here we give the outline of proof. See Appendix A for the detailed computation. 
		
		\medskip
		\noindent For simplicity we omit the $q$-shuffle product symbol $\star$ in this proof. 
		
		\medskip
		\noindent We first compute the left-hand side of \eqref{eq:RKRKind1result}. Substitute the terms using \eqref{eq:Rrecur1}, \eqref{eq:hRrecur1}, \eqref{eq:Krecur}, then apply \eqref{eq:RE}, \eqref{eq:EFR}, \eqref{eq:hRhRR1}, and a variant of \eqref{eq:RKRKind1.1}. This yields 
		\begin{align*}
			&R^{(j_1+\frac{1}{2},j_2)}(t/s) K^{(j_1+\frac{1}{2})}_{1}(s) \hR^{(j_1+\frac{1}{2},j_2)} K^{(j_2)}_{2}(t)\\
			&=\cF^{(j_1+\frac{1}{2})}_{12}R^{(\frac{1}{2},j_2)}_{13}(q^{-j_1}t/s)R^{(j_1,j_2)}_{23}(q^{\frac{1}{2}}t/s) K^{(\frac{1}{2})}_{1}(q^{j_1}s)\hR^{(\frac{1}{2},j_1)}_{12}K^{(j_1)}_{2}(q^{-\frac{1}{2}}s)\\
			&\hspace{14em} \hR^{(\frac{1}{2},j_2)}_{13}\hR^{(j_1,j_2)}_{23} K^{(j_2)}_{3}(t) \cE^{(j_1+\frac{1}{2})}_{12}. 
		\end{align*}
		
		\noindent We next compute the right-hand side of \eqref{eq:RKRKind1result}. By using \eqref{eq:Rrecur1}, \eqref{eq:hRrecur1}, \eqref{eq:Krecur} and then applying \eqref{eq:RE}, \eqref{eq:EFR}, \eqref{eq:RRR2}, \eqref{eq:RKRKind1.1}, we obtain  
		\begin{align*}
			&K^{(j_2)}_{2}(t) \hR^{(j_1+\frac{1}{2},j_2)} K^{(j_1+\frac{1}{2})}_{1}(s) R^{(j_1+\frac{1}{2},j_2)}(t/s)\\
			&=\cF^{(j_1+\frac{1}{2})}_{12}K^{(j_2)}_{3}(t)\hR^{(\frac{1}{2},j_2)}_{13}\hR^{(j_1,j_2)}_{23}K^{(\frac{1}{2})}_{1}(q^{j_1}s)\hR^{(\frac{1}{2},j_1)}_{12}\\
			&\hspace{8em} K^{(j_1)}_{2}(q^{-\frac{1}{2}}s)R^{(\frac{1}{2},j_2)}_{13}(q^{-j_1}t/s)R^{(j_1,j_2)}_{23}(q^{\frac{1}{2}}t/s)\cE^{(j_1+\frac{1}{2})}_{12}. 
		\end{align*}
		
		\noindent Using \eqref{eq:hRhRR2}, \eqref{eq:hRhRR3}, \eqref{eq:RKRKind1.2}, \eqref{eq:RKRKind1.3} we verify that both sides of \eqref{eq:RKRKind1result} are equal. 
	\end{proof}
	
	\begin{lemma}\label{lem:RKRKind2}\rm
		Given $j_1,j_2 \in \frac{1}{2}\mN^+$. Suppose we have \\
		\begin{equation}\label{eq:RKRKind2.1}
			R^{(j_1,\frac{1}{2})}(t/s) \star K^{(j_1)}_{1}(s) \star \hR^{(j_1,\frac{1}{2})} \star K^{(\frac{1}{2})}_{2}(t)=K^{(\frac{1}{2})}_{2}(t) \star \hR^{(j_1,\frac{1}{2})} \star K^{(j_1)}_{1}(s) \star R^{(j_1,\frac{1}{2})}(t/s), 
		\end{equation}
		\begin{equation}\label{eq:RKRKind2.2}
			R^{(\frac{1}{2},j_2)}(t/s) \star K^{(\frac{1}{2})}_{1}(s) \star \hR^{(\frac{1}{2},j_2)} \star K^{(j_2)}_{2}(t)=K^{(j_2)}_{2}(t) \star \hR^{(\frac{1}{2},j_2)} \star K^{(\frac{1}{2})}_{1}(s) \star R^{(\frac{1}{2},j_2)}(t/s), 
		\end{equation}
		\begin{equation}\label{eq:RKRKind2.3}
			R^{(j_1,j_2)}(t/s) \star K^{(j_1)}_{1}(s) \star \hR^{(j_1,j_2)} \star K^{(j_2)}_{2}(t)=K^{(j_2)}_{2}(t) \star \hR^{(j_1,j_2)} \star K^{(j_1)}_{1}(s) \star R^{(j_1,j_2)}(t/s). 
		\end{equation}
		Then 
		\begin{align}\label{eq:RKRKind2result}
			\begin{split}
				&R^{(j_1,j_2+\frac{1}{2})}(t/s) \star K^{(j_1)}_{1}(s) \star \hR^{(j_1,j_2+\frac{1}{2})} \star K^{(j_2+\frac{1}{2})}_{2}(t)\\
				&\hspace{10em}=K^{(j_2+\frac{1}{2})}_{2}(t) \star \hR^{(j_1,j_2+\frac{1}{2})} \star K^{(j_1)}_{1}(s) \star R^{(j_1,j_2+\frac{1}{2})}(t/s). 
			\end{split}
		\end{align}
	\end{lemma}
	\begin{proof}
		Again we give the outline of proof. The detailed computation is similar to that of Lemma \ref{lem:RKRKind1} and thus omitted. 
		
		\medskip
		\noindent For simplicity we omit the $q$-shuffle product symbol $\star$ in this proof. 
		
		\medskip		
		\noindent We first compute the left-hand side of \eqref{eq:RKRKind2result}. Substitute the terms using \eqref{eq:Rrecur1}, \eqref{eq:hRrecur1}, \eqref{eq:Krecur}, then apply \eqref{eq:RE}, \eqref{eq:EFR}, \eqref{eq:hRhRR5}, and a variant of \eqref{eq:RKRKind2.2}. This yields 
		\begin{align*}
			&R^{(j_1,j_2+\frac{1}{2})}(t/s) K^{(j_1)}_{1}(s) \hR^{(j_1,j_2+\frac{1}{2})} K^{(j_2+\frac{1}{2})}_{2}(t)\\
			&=\cF^{(j_2+\frac{1}{2})}_{23}R^{(j_1,j_2)}_{13}(q^{-\frac{1}{2}}t/s)R^{(j_1,\frac{1}{2})}_{12}(q^{j_2}t/s)K^{(j_1)}_{1}(s)\hR^{(j_1,\frac{1}{2})}_{12}\hR^{(j_1,j_2)}_{13} \\
			&\hspace{14em} K^{(\frac{1}{2})}_{2}(q^{j_2}t)\hR^{(\frac{1}{2},j_2)}_{23}K^{(j_2)}_{3}(q^{-\frac{1}{2}}t)\cE^{(j_2+\frac{1}{2})}_{23}. 
		\end{align*}
		
		\noindent We next compute the right-hand side of \eqref{eq:RKRKind2result}. By using \eqref{eq:Rrecur1}, \eqref{eq:hRrecur1}, \eqref{eq:Krecur} and then applying \eqref{eq:RE}, \eqref{eq:EFR}, \eqref{eq:RRR3}, \eqref{eq:hRhRR5}, we obtain 
		\begin{align*}
			&K^{(j_2+\frac{1}{2})}_{2}(t) \hR^{(j_1,j_2+\frac{1}{2})} K^{(j_1)}_{1}(s) R^{(j_1,j_2+\frac{1}{2})}(t/s) \\
			&=\cF^{(j_2+\frac{1}{2})}_{23}K^{(\frac{1}{2})}_{2}(q^{j_2}t)\hR^{(\frac{1}{2},j_2)}_{23}K^{(j_2)}_{3}(q^{-\frac{1}{2}}t)\hR^{(j_1,\frac{1}{2})}_{12}\hR^{(j_1,j_2)}_{13} \\
			&\hspace{8em} K^{(j_1)}_{1}(s)R^{(j_1,j_2)}_{13}(q^{-\frac{1}{2}}t/s)R^{(j_1,\frac{1}{2})}_{12}(q^{j_2}t/s)\cE^{(j_2+\frac{1}{2})}_{23}. 
		\end{align*}
		
		\noindent Using \eqref{eq:hRhRR4}, \eqref{eq:hRhRR3}, \eqref{eq:RKRKind2.1}, \eqref{eq:RKRKind2.3} we verify that both sides of \eqref{eq:RKRKind2result} are equal. 
	\end{proof}
	
	\noindent Now we are ready to show Theorem \ref{thm:RKRK} by induction. 
	
	\begin{proof}[Proof of Theorem \ref{thm:RKRK}]
		By \eqref{eq:FM1/2}, we have that \eqref{eq:RKRK} holds when $j_1=j_2=\frac{1}{2}$. 
		
		\medskip
		\noindent Setting $j_1=\frac{1}{2}$ in Lemma \ref{lem:RKRKind2} and using induction on $j_2$, we have that \eqref{eq:RKRK} holds when $j_1=\frac{1}{2}$ and $j_2 \in \frac{1}{2}\mN^+$. 
		
		\medskip
		\noindent Now, using Lemma \ref{lem:RKRKind1} and induction on $j_1$, we have that \eqref{eq:RKRK} holds when $j_1,j_2 \in \frac{1}{2}\mN^+$. 
	\end{proof}
	
	\noindent Corollary \ref{cor:KRKR} is a straightforward consequence of Theorem \ref{thm:RKRK}. 
	
	\begin{proof}[Proof of Corollary \ref{cor:KRKR}]
		Multiply both sides of \eqref{eq:RKRK} by $R^{(j_1,j_2)}(s/t)$ on the left and on the right and simplify the result using Lemma \ref{lem:RR}. This yields \eqref{eq:KRKR}. 
	\end{proof}
	
	\section{Acknowledgements}
	\noindent The author would like to gratefully acknowledge that this research is motivated by discussions with Pascal Baseilhac, who suggested the possible connection from the author's previous work to the work of Baseilhac and to the work of Lemarthe, Baseilhac, and Gainutdinov. The author would like to thank Bart Vlaar for pointing out the alternative approach in Appendix D, and for suggesting Remark \ref{rm:D1} and Conjecture \ref{conj:D2}. The author would also like to thank Paul Terwilliger for various useful discussions. 
	
	\medskip
	\noindent This research was supported by National Natural Science Foundation of China (Grant No.\ 12501032) and Beijing Natural Science Foundation (Grant No.\ IS24003). 
	
	\appendix
	\section{Detailed computations of two proofs}
	\noindent In this section, we give the proofs of Proposition \ref{prop:Krecur} and Lemma \ref{lem:RKRKind1} with detailed computations. 
	
	\begin{proof}[Proof of Proposition \ref{prop:Krecur}]
		We first clarify some abuse of notation in this proof. There will be some undefined terms for certain values of $a,b$. For example, the term $K^{(j)}_{(a,b-1)}$ is not defined when $b=1$. As we will see, each undefined term is always multiplied by a zero and thus does not impair the proof. 
		
		\medskip
		\noindent For $1 \leq a,b \leq 2j+2$, the $(a,b)$-entry of the right-hand side of \eqref{eq:Krecur} is equal to 
		\begin{align*}
			&q^{j+1-a}\frac{([2j+2-a]_q[2j+2-b]_q)^\frac{1}{2}}{[2j+2-a]_q+[a-1]_q}K^{(\frac{1}{2})}_{(1,1)}(q^{j}t) \star K^{(j)}_{(a,b)}(q^{-\frac{1}{2}}t)\\
			+&q^{a-j-1}\frac{([2j+2-a]_q[b-1]_q)^\frac{1}{2}}{[2j+2-a]_q+[a-1]_q}K^{(\frac{1}{2})}_{(1,2)}(q^{j}t) \star K^{(j)}_{(a,b-1)}(q^{-\frac{1}{2}}t)\\
			+&q^{j+2-a}\frac{([a-1]_q[2j+2-b]_q)^\frac{1}{2}}{[2j+2-a]_q+[a-1]_q}K^{(\frac{1}{2})}_{(2,1)}(q^{j}t) \star K^{(j)}_{(a-1,b)}(q^{-\frac{1}{2}}t)\\
			+&q^{a-j-2}\frac{([a-1]_q[b-1]_q)^\frac{1}{2}}{[2j+2-a]_q+[a-1]_q}K^{(\frac{1}{2})}_{(2,2)}(q^{j}t) \star K^{(j)}_{(a-1,b-1)}(q^{-\frac{1}{2}}t).  
		\end{align*}
		
		\noindent Applying \eqref{eq:K}, \eqref{eq:Kalt}, we obtain 
		\begin{align*}
			&q^{3j-\frac{3}{2}a+\frac{1}{2}b+2}\frac{([2j+2-a]_q[2j+2-b]_q)^\frac{1}{2}}{[2j+2-a]_q+[a-1]_q}\varphi(a,b,j)\\
			&\hspace{14em} t^{a-b-2j+1}W^-(q^{2j}t^2) \star x^{1-b}\Delta^{(-2j)}(-q^{-1}t^2)y^{a-2j-1}\\
			+&q^{\frac{1}{2}a+\frac{1}{2}b-\frac{3}{2}}\frac{([2j+2-a]_q[b-1]_q)^\frac{1}{2}}{[2j+2-a]_q+[a-1]_q}\varphi(a,b-1,j)\\
			&\hspace{14em} t^{a-b-2j+1}G(q^{2j}t^2) \star x^{2-b}\Delta^{(-2j)}(-q^{-1}t^2)y^{a-2j-1}\\
			+&q^{2j-\frac{1}{2}a-\frac{1}{2}b+\frac{3}{2}}\frac{([a-1]_q[2j+2-b]_q)^\frac{1}{2}}{[2j+2-a]_q+[a-1]_q}\psi(a-1,b,j)\\
			&\hspace{14em} t^{b-a-2j+1}\tG(q^{2j}t^2) \star y^{b-2j-1}\tDelta^{(-2j)}(-q^{-1}t^2)x^{2-a}\\
			+&q^{j+\frac{3}{2}a-\frac{1}{2}b-1}\frac{([a-1]_q[b-1]_q)^\frac{1}{2}}{[2j+2-a]_q+[a-1]_q}\psi(a-1,b-1,j)\\
			&\hspace{14em} t^{b-a-2j+1}W^+(q^{2j}t^2) \star y^{b-2j-2}\tDelta^{(-2j)}(-q^{-1}t^2)x^{2-a}. 
		\end{align*}
		
		\noindent Using \eqref{eq:f}, \eqref{eq:g} to compute the coefficients, we obtain 
		\begin{align*}
			&q^{2b-2}[2j+1]_q\frac{[2j+2-a]_q}{[2j+2-a]_q+[a-1]_q}\varphi(a,b,j+\tfrac{1}{2})\\
			&\hspace{14em} t^{a-b-2j+1}W^-(q^{2j}t^2) \star x^{1-b}\Delta^{(-2j)}(-q^{-1}t^2)y^{a-2j-1}\\
			+&q^{b-2}[2j+1]_q\frac{[b-1]_q[2j+2-a]_q}{[2j+2-a]_q+[a-1]_q}\varphi(a,b,j+\tfrac{1}{2})\\
			&\hspace{14em} t^{a-b-2j+1}G(q^{2j}t^2) \star x^{2-b}\Delta^{(-2j)}(-q^{-1}t^2)y^{a-2j-1}\\
			+&q^{2j-b+1}[2j+1]_q\frac{[a-1]_q[2j+2-b]_q}{[2j+2-a]_q+[a-1]_q}\psi(a,b,j+\tfrac{1}{2})\\
			&\hspace{14em} t^{b-a-2j+1}\tG(q^{2j}t^2) \star y^{b-2j-1}\tDelta^{(-2j)}(-q^{-1}t^2)x^{2-a}\\
			+&q^{4j-2b+4}[2j+1]_q\frac{[a-1]_q}{[2j+2-a]_q+[a-1]_q}\psi(a,b,j+\tfrac{1}{2})\\
			&\hspace{14em} t^{b-a-2j+1}W^+(q^{2j}t^2) \star y^{b-2j-2}\tDelta^{(-2j)}(-q^{-1}t^2)x^{2-a}, 
		\end{align*}
		
		\noindent Now we apply Lemma \ref{lem:Drecur} and obtain 
		\begin{align*}
			&\frac{[2j+2-a]_q}{[2j+2-a]_q+[a-1]_q}\varphi(a,b,j+\tfrac{1}{2})t^{a-b-2j-1}x^{1-b}\Delta^{(-2j-1)}(-t^2)y^{a-2j-2}\\
			+&\frac{[a-1]_q}{[2j+2-a]_q+[a-1]_q}\psi(a,b,j+\tfrac{1}{2})t^{b-a-2j-1}y^{b-2j-2}\tDelta^{(-2j-1)}(-t^2)x^{1-a}. 
		\end{align*}
		
		\noindent By \eqref{eq:K}, \eqref{eq:Kalt}, this is equal to $K^{(j+\frac{1}{2})}_{(a,b)}(t)$ as desired. 
	\end{proof}
	
	\begin{proof}[Proof of Lemma \ref{lem:RKRKind1}]
		Multiply both sides of \eqref{eq:RKRKind1.1} by $R^{(\frac{1}{2},j_1)}(s/t)$ on the left and on the right and simplify the result using Lemma \ref{lem:RR}. This yields 
		\begin{equation}\label{eq:RKRKind1.1altappendix}
			K^{(\frac{1}{2})}_{1}(s) \star \hR^{(\frac{1}{2},j_1)} \star K^{(j_1)}_{2}(t) \star R^{(\frac{1}{2},j_1)}(s/t)=R^{(\frac{1}{2},j_1)}(s/t) \star K^{(j_1)}_{2}(t) \star \hR^{(\frac{1}{2},j_1)} \star K^{(\frac{1}{2})}_{1}(s). 
		\end{equation}
		
		\noindent Below each underlined part is computed using the commented result. For simplicity, we omit the $q$-shuffle product symbol $\star$ for the rest of this proof. 
		
		\medskip
		\noindent We first compute the left-hand side of \eqref{eq:RKRKind1result}. 
		\begin{align*}
			&R^{(j_1+\frac{1}{2},j_2)}(t/s) K^{(j_1+\frac{1}{2})}_{1}(s) \hR^{(j_1+\frac{1}{2},j_2)} K^{(j_2)}_{2}(t)\\
			&=\underset{\eqref{eq:Rrecur1}}{\underline{R^{(j_1+\frac{1}{2},j_2)}(t/s)}} \underset{\eqref{eq:Krecur}}{\underline{K^{(j_1+\frac{1}{2})}_{12}(s)}} \underset{\eqref{eq:hRrecur1}}{\underline{\hR^{(j_1+\frac{1}{2},j_2)}}} K^{(j_2)}_{3}(t)\\
			&=\cF^{(j_1+\frac{1}{2})}_{12}R^{(\frac{1}{2},j_2)}_{13}(q^{-j_1}t/s)R^{(j_1,j_2)}_{23}(q^{\frac{1}{2}}t/s)\cE^{(j_1+\frac{1}{2})}_{12}\cF^{(j_1+\frac{1}{2})}_{12} K^{(\frac{1}{2})}_{1}(q^{j_1}s) \hR^{(\frac{1}{2},j_1)}_{12} K^{(j_1)}_{2}(q^{-\frac{1}{2}}s) \\
			&\hspace{8em} \cE^{(j_1+\frac{1}{2})}_{12}\cF^{(j_1+\frac{1}{2})}_{12}\hR^{(\frac{1}{2},j_2)}_{13}\hR^{(j_1,j_2)}_{23} \underset{\eqref{eq:RE}}{\underline{\cE^{(j_1+\frac{1}{2})}_{12} \cH^{(j_1)}_{12}}}\left(\cH^{(j_1)}_{12}\right)^{-1}K^{(j_2)}_{3}(t)\\
			&=\cF^{(j_1+\frac{1}{2})}_{12}R^{(\frac{1}{2},j_2)}_{13}(q^{-j_1}t/s)R^{(j_1,j_2)}_{23}(q^{\frac{1}{2}}t/s)\cE^{(j_1+\frac{1}{2})}_{12}\cF^{(j_1+\frac{1}{2})}_{12} K^{(\frac{1}{2})}_{1}(q^{j_1}s) \hR^{(\frac{1}{2},j_1)}_{12} K^{(j_1)}_{2}(q^{-\frac{1}{2}}s) \\
			&\hspace{8em} \cE^{(j_1+\frac{1}{2})}_{12}\cF^{(j_1+\frac{1}{2})}_{12} \underset{\eqref{eq:hRhRR1}}{\underline{\hR^{(\frac{1}{2},j_2)}_{13}\hR^{(j_1,j_2)}_{23} R^{(\frac{1}{2},j_1)}_{12}(q^{j_1+\frac{1}{2}})}} \cE^{(j_1+\frac{1}{2})}_{12} \left(\cH^{(j_1)}_{12}\right)^{-1}K^{(j_2)}_{3}(t)\\
			&=\cF^{(j_1+\frac{1}{2})}_{12}R^{(\frac{1}{2},j_2)}_{13}(q^{-j_1}t/s)R^{(j_1,j_2)}_{23}(q^{\frac{1}{2}}t/s)\cE^{(j_1+\frac{1}{2})}_{12}\cF^{(j_1+\frac{1}{2})}_{12} K^{(\frac{1}{2})}_{1}(q^{j_1}s) \hR^{(\frac{1}{2},j_1)}_{12} K^{(j_1)}_{2}(q^{-\frac{1}{2}}s) \\
			&\hspace{8em} \underset{\eqref{eq:EFR}}{\underline{\cE^{(j_1+\frac{1}{2})}_{12}\cF^{(j_1+\frac{1}{2})}_{12} R^{(\frac{1}{2},j_1)}_{12}(q^{j_1+\frac{1}{2}})}} \hR^{(j_1,j_2)}_{23}\hR^{(\frac{1}{2},j_2)}_{13}  \cE^{(j_1+\frac{1}{2})}_{12} \left(\cH^{(j_1)}_{12}\right)^{-1}K^{(j_2)}_{3}(t)\\
			&=\cF^{(j_1+\frac{1}{2})}_{12}R^{(\frac{1}{2},j_2)}_{13}(q^{-j_1}t/s)R^{(j_1,j_2)}_{23}(q^{\frac{1}{2}}t/s)\cE^{(j_1+\frac{1}{2})}_{12}\cF^{(j_1+\frac{1}{2})}_{12} \underset{\eqref{eq:RKRKind1.1altappendix}}{\underline{K^{(\frac{1}{2})}_{1}(q^{j_1}s) \hR^{(\frac{1}{2},j_1)}_{12} K^{(j_1)}_{2}(q^{-\frac{1}{2}}s)R^{(\frac{1}{2},j_1)}_{12}(q^{j_1+\frac{1}{2}})}} \\
			&\hspace{20em} \hR^{(j_1,j_2)}_{23}\hR^{(\frac{1}{2},j_2)}_{13}  \cE^{(j_1+\frac{1}{2})}_{12} \left(\cH^{(j_1)}_{12}\right)^{-1}K^{(j_2)}_{3}(t)\\
			&=\cF^{(j_1+\frac{1}{2})}_{12}R^{(\frac{1}{2},j_2)}_{13}(q^{-j_1}t/s)R^{(j_1,j_2)}_{23}(q^{\frac{1}{2}}t/s) \underset{\eqref{eq:EFR}}{\underline{\cE^{(j_1+\frac{1}{2})}_{12}\cF^{(j_1+\frac{1}{2})}_{12} R^{(\frac{1}{2},j_1)}_{12}(q^{j_1+\frac{1}{2}})}} K^{(j_1)}_{2}(q^{-\frac{1}{2}}s)\hR^{(\frac{1}{2},j_1)}_{12}K^{(\frac{1}{2})}_{1}(q^{j_1}s) \\
			&\hspace{20em} \hR^{(j_1,j_2)}_{23}\hR^{(\frac{1}{2},j_2)}_{13}  \cE^{(j_1+\frac{1}{2})}_{12} \left(\cH^{(j_1)}_{12}\right)^{-1}K^{(j_2)}_{3}(t)\\
			&=\cF^{(j_1+\frac{1}{2})}_{12}R^{(\frac{1}{2},j_2)}_{13}(q^{-j_1}t/s)R^{(j_1,j_2)}_{23}(q^{\frac{1}{2}}t/s) \underset{\eqref{eq:RKRKind1.1altappendix}}{\underline{R^{(\frac{1}{2},j_1)}_{12}(q^{j_1+\frac{1}{2}})K^{(j_1)}_{2}(q^{-\frac{1}{2}}s)\hR^{(\frac{1}{2},j_1)}_{12}K^{(\frac{1}{2})}_{1}(q^{j_1}s)}} \\
			&\hspace{20em} \hR^{(j_1,j_2)}_{23}\hR^{(\frac{1}{2},j_2)}_{13} \cE^{(j_1+\frac{1}{2})}_{12} \left(\cH^{(j_1)}_{12}\right)^{-1}K^{(j_2)}_{3}(t)\\
			&=\cF^{(j_1+\frac{1}{2})}_{12}R^{(\frac{1}{2},j_2)}_{13}(q^{-j_1}t/s)R^{(j_1,j_2)}_{23}(q^{\frac{1}{2}}t/s) K^{(\frac{1}{2})}_{1}(q^{j_1}s)\hR^{(\frac{1}{2},j_1)}_{12}K^{(j_1)}_{2}(q^{-\frac{1}{2}}s) \\
			&\hspace{14em} \underset{\eqref{eq:hRhRR1}}{\underline{R^{(\frac{1}{2},j_1)}_{12}(q^{j_1+\frac{1}{2}})\hR^{(j_1,j_2)}_{23}\hR^{(\frac{1}{2},j_2)}_{13}}} \cE^{(j_1+\frac{1}{2})}_{12} \left(\cH^{(j_1)}_{12}\right)^{-1}K^{(j_2)}_{3}(t)\\
			&=\cF^{(j_1+\frac{1}{2})}_{12}R^{(\frac{1}{2},j_2)}_{13}(q^{-j_1}t/s)R^{(j_1,j_2)}_{23}(q^{\frac{1}{2}}t/s) K^{(\frac{1}{2})}_{1}(q^{j_1}s)\hR^{(\frac{1}{2},j_1)}_{12}K^{(j_1)}_{2}(q^{-\frac{1}{2}}s) \\
			&\hspace{14em} \hR^{(\frac{1}{2},j_2)}_{13}\hR^{(j_1,j_2)}_{23} \underset{\eqref{eq:RE}}{\underline{R^{(\frac{1}{2},j_1)}_{12}(q^{j_1+\frac{1}{2}}) \cE^{(j_1+\frac{1}{2})}_{12}}} \left(\cH^{(j_1)}_{12}\right)^{-1}K^{(j_2)}_{3}(t)\\
			&=\cF^{(j_1+\frac{1}{2})}_{12}R^{(\frac{1}{2},j_2)}_{13}(q^{-j_1}t/s)R^{(j_1,j_2)}_{23}(q^{\frac{1}{2}}t/s) K^{(\frac{1}{2})}_{1}(q^{j_1}s)\hR^{(\frac{1}{2},j_1)}_{12}K^{(j_1)}_{2}(q^{-\frac{1}{2}}s) \\
			&\hspace{14em} \hR^{(\frac{1}{2},j_2)}_{13}\hR^{(j_1,j_2)}_{23} \underset{\text{explained below}}{\underline{\cE^{(j_1+\frac{1}{2})}_{12} K^{(j_2)}_{3}(t)}} \\
			&=\cF^{(j_1+\frac{1}{2})}_{12}R^{(\frac{1}{2},j_2)}_{13}(q^{-j_1}t/s)R^{(j_1,j_2)}_{23}(q^{\frac{1}{2}}t/s) K^{(\frac{1}{2})}_{1}(q^{j_1}s)\hR^{(\frac{1}{2},j_1)}_{12}K^{(j_1)}_{2}(q^{-\frac{1}{2}}s)\\
			&\hspace{14em} \hR^{(\frac{1}{2},j_2)}_{13}\hR^{(j_1,j_2)}_{23} K^{(j_2)}_{3}(t) \cE^{(j_1+\frac{1}{2})}_{12}. 
		\end{align*}
		
		\noindent Note that $\cE^{(j_1+\frac{1}{2})}$ is a $(4j_1+2) \times (2j_1+2)$ matrix. In the above computation, we interpret the second-to-last $K^{(j_2)}_{3}(t)$ as $\mI_{2j_1+2} \otimes K^{(j_2)}(t)$ and the last $K^{(j_2)}_{3}(t)$ as $\mI_{4j_1+2} \otimes K^{(j_2)}(t)$, then the last step follows by 
		\begin{align*}
			&\cE^{(j_1+\frac{1}{2})}_{12} K^{(j_2)}_{3}(t)=\big(\cE^{(j_1+\frac{1}{2})} \otimes \mI_{2j_2+1}\big)\big(\mI_{2j_1+2} \otimes K^{(j_2)}(t)\big)=\cE^{(j_1+\frac{1}{2})} \otimes K^{(j_2)}(t)\\
			&=\big(\mI_{4j_1+2} \otimes K^{(j_2)}(t)\big)\big(\cE^{(j_1+\frac{1}{2})} \otimes \mI_{2j_2+1}\big)=K^{(j_2)}_{3}(t) \cE^{(j_1+\frac{1}{2})}_{12}. 
		\end{align*}
		
		\noindent We next compute the right-hand side of \eqref{eq:RKRKind1result} in a similar way. Substituting the terms using \eqref{eq:Rrecur1}, \eqref{eq:hRrecur1}, \eqref{eq:Krecur} and then applying \eqref{eq:RE}, \eqref{eq:EFR}, \eqref{eq:RRR2}, \eqref{eq:RKRKind1.1altappendix}, we obtain  
		\begin{align*}
			&K^{(j_2)}_{2}(t) \star \hR^{(j_1+\frac{1}{2},j_2)} \star K^{(j_1+\frac{1}{2})}_{1}(s) \star R^{(j_1+\frac{1}{2},j_2)}(t/s)\\
			&=\cF^{(j_1+\frac{1}{2})}_{12}K^{(j_2)}_{3}(t)\hR^{(\frac{1}{2},j_2)}_{13}\hR^{(j_1,j_2)}_{23}K^{(\frac{1}{2})}_{1}(q^{j_1}s)\hR^{(\frac{1}{2},j_1)}_{12}\\
			&\hspace{8em} K^{(j_1)}_{2}(q^{-\frac{1}{2}}s)R^{(\frac{1}{2},j_2)}_{13}(q^{-j_1}t/s)R^{(j_1,j_2)}_{23}(q^{\frac{1}{2}}t/s)\cE^{(j_1+\frac{1}{2})}_{12}. 
		\end{align*}
		
		\noindent Comparing the above results about boths sides of \eqref{eq:RKRKind1result}, it suffices to show that 
		\begin{align*}
			&R^{(\frac{1}{2},j_2)}_{13}(q^{-j_1}t/s)R^{(j_1,j_2)}_{23}(q^{\frac{1}{2}}t/s) K^{(\frac{1}{2})}_{1}(q^{j_1}s)\hR^{(\frac{1}{2},j_1)}_{12}K^{(j_1)}_{2}(q^{-\frac{1}{2}}s)\hR^{(\frac{1}{2},j_2)}_{13}\hR^{(j_1,j_2)}_{23}K^{(j_2)}_{3}(t)\\
			&=K^{(j_2)}_{3}(t)\hR^{(\frac{1}{2},j_2)}_{13}\hR^{(j_1,j_2)}_{23}K^{(\frac{1}{2})}_{1}(q^{j_1}s)\hR^{(\frac{1}{2},j_1)}_{12}K^{(j_1)}_{2}(q^{-\frac{1}{2}}s)R^{(\frac{1}{2},j_2)}_{13}(q^{-j_1}t/s)R^{(j_1,j_2)}_{23}(q^{\frac{1}{2}}t/s). 
		\end{align*}
		
		\noindent This is verified as follows. 
		\begin{align*}
			&R^{(\frac{1}{2},j_2)}_{13}(q^{-j_1}t/s){\underline{R^{(j_1,j_2)}_{23}(q^{\frac{1}{2}}t/s) K^{(\frac{1}{2})}_{1}(q^{j_1}s)}}\hR^{(\frac{1}{2},j_1)}_{12}{\underline{K^{(j_1)}_{2}(q^{-\frac{1}{2}}s)\hR^{(\frac{1}{2},j_2)}_{13}}}\hR^{(j_1,j_2)}_{23}K^{(j_2)}_{3}(t)\\
			&=R^{(\frac{1}{2},j_2)}_{13}(q^{-j_1}t/s)K^{(\frac{1}{2})}_{1}(q^{j_1}s)\underset{\eqref{eq:hRhRR2}}{\underline{R^{(j_1,j_2)}_{23}(q^{\frac{1}{2}}t/s)\hR^{(\frac{1}{2},j_1)}_{12}\hR^{(\frac{1}{2},j_2)}_{13}}}K^{(j_1)}_{2}(q^{-\frac{1}{2}}s)\hR^{(j_1,j_2)}_{23}K^{(j_2)}_{3}(t)\\
			&=R^{(\frac{1}{2},j_2)}_{13}(q^{-j_1}t/s)K^{(\frac{1}{2})}_{1}(q^{j_1}s)\hR^{(\frac{1}{2},j_2)}_{13}\hR^{(\frac{1}{2},j_1)}_{12}\underset{\eqref{eq:RKRKind1.3}}{\underline{R^{(j_1,j_2)}_{23}(q^{\frac{1}{2}}t/s)K^{(j_1)}_{2}(q^{-\frac{1}{2}}s)\hR^{(j_1,j_2)}_{23}K^{(j_2)}_{3}(t)}}\\
			&=R^{(\frac{1}{2},j_2)}_{13}(q^{-j_1}t/s)K^{(\frac{1}{2})}_{1}(q^{j_1}s)\hR^{(\frac{1}{2},j_2)}_{13}{\underline{\hR^{(\frac{1}{2},j_1)}_{12}K^{(j_2)}_{3}(t)}}\hR^{(j_1,j_2)}_{23}K^{(j_1)}_{2}(q^{-\frac{1}{2}}s)R^{(j_1,j_2)}_{23}(q^{\frac{1}{2}}t/s)\\
			&=\underset{\eqref{eq:RKRKind1.2}}{\underline{R^{(\frac{1}{2},j_2)}_{13}(q^{-j_1}t/s)K^{(\frac{1}{2})}_{1}(q^{j_1}s)\hR^{(\frac{1}{2},j_2)}_{13}K^{(j_2)}_{3}(t)}}\hR^{(\frac{1}{2},j_1)}_{12}\hR^{(j_1,j_2)}_{23}K^{(j_1)}_{2}(q^{-\frac{1}{2}}s)R^{(j_1,j_2)}_{23}(q^{\frac{1}{2}}t/s)\\
			&=K^{(j_2)}_{3}(t)\hR^{(\frac{1}{2},j_2)}_{13}K^{(\frac{1}{2})}_{1}(q^{j_1}s)\underset{\eqref{eq:hRhRR3}}{\underline{R^{(\frac{1}{2},j_2)}_{13}(q^{-j_1}t/s)\hR^{(\frac{1}{2},j_1)}_{12}\hR^{(j_1,j_2)}_{23}}}K^{(j_1)}_{2}(q^{-\frac{1}{2}}s)R^{(j_1,j_2)}_{23}(q^{\frac{1}{2}}t/s)\\
			&=K^{(j_2)}_{3}(t)\hR^{(\frac{1}{2},j_2)}_{13}{\underline{K^{(\frac{1}{2})}_{1}(q^{j_1}s)\hR^{(j_1,j_2)}_{23}}}\hR^{(\frac{1}{2},j_1)}_{12}{\underline{R^{(\frac{1}{2},j_2)}_{13}(q^{-j_1}t/s)K^{(j_1)}_{2}(q^{-\frac{1}{2}}s)}}R^{(j_1,j_2)}_{23}(q^{\frac{1}{2}}t/s)\\
			&=K^{(j_2)}_{3}(t)\hR^{(\frac{1}{2},j_2)}_{13}\hR^{(j_1,j_2)}_{23}K^{(\frac{1}{2})}_{1}(q^{j_1}s)\hR^{(\frac{1}{2},j_1)}_{12}K^{(j_1)}_{2}(q^{-\frac{1}{2}}s)R^{(\frac{1}{2},j_2)}_{13}(q^{-j_1}t/s)R^{(j_1,j_2)}_{23}(q^{\frac{1}{2}}t/s). 
		\end{align*}
		
		\noindent Therefore, we have proved \eqref{eq:RKRKind1result}. 
	\end{proof}
	
	\section{Entries of the R-matrix}
	\noindent In this appendix, we show some results about the entries of the R-matrix $R^{(j_1,j_2)}$. 
	
	\medskip
	\noindent Recall that in Definition \ref{def:R} we gave a closed form for $R^{(\frac{1}{2},\frac{1}{2})}$ and defined $R^{(j_1,j_2)}$ recursively for $j_1,j_2 \in \frac{1}{2}\mN^+$. We first give a closed form for $R^{(\frac{1}{2},j_2)}$ for $j_2 \in \frac{1}{2}\mN^+$. 
	
	\begin{proposition}\label{prop:R^1/2}\rm
		For $j \in \frac{1}{2}\mN^+$, all the nonzero entries of $R^{(\frac{1}{2},j)}(t)$ are given as follows: 
		\begin{align*}
			&R^{(\frac{1}{2},j)}_{(a,a)}(t)=R^{(\frac{1}{2},j)}_{(4j+3-a,4j+3-a)}(t)=c(q^{j+\frac{3}{2}-a}t) \prod_{k=0}^{2j-2}c(q^{j-\frac{1}{2}-k}t) & \hspace{1em} (1 \leq a \leq 2j+1); \\
			&R^{(\frac{1}{2},j)}_{(a,a+2j)}(t)=R^{(\frac{1}{2},j)}_{(a+2j,a)}(t)=c(q)\big([2j+2-a]_q[a-1]_q\big)^{\frac{1}{2}} \prod_{k=0}^{2j-2}c(q^{j-\frac{1}{2}-k}t) & \hspace{1em} (2 \leq a \leq 2j+1). 
		\end{align*}
	\end{proposition}
	\begin{proof}
		Follows from \eqref{eq:R^1/2}, \eqref{eq:Rdef1} by direct computation. 
	\end{proof}
	
	\noindent The matrix $R^{(j_1,j_2)}$ has no known closed form. We give a result on the possible location of its nonzero entries. 
	
	\medskip
	\noindent We call a square matrix a \textit{$c$-diagonal matrix} if its $(a,b)$-entry is zero whenever $b-a \neq c$. For example, a $0$-diagonal matrix is a diagonal matrix; an $1$-diagonal matrix has zero entries outside the superdiagonal; a $-1$-diagonal matrix has zero entries outside the subdiagonal. 
	
	\begin{proposition}\label{prop:Rnonzero}\rm
		For $j_1,j_2 \in \frac{1}{2}\mN^+$, we write $R^{(j_1,j_2)}(t)$ as a $(2j_1+1) \times (2j_1+1)$ block matrix where each block is of size $(2j_2+1) \times (2j_2+1)$. Then the $(a,b)$-block is an $(a-b)$-diagonal matrix if $|a-b| \leq 2j_2$ and is zero if $|a-b| \geq 2j_2+1$. 
	\end{proposition}
	\begin{proof}
		We use induction on $j_1$. 
		
		\medskip
		\noindent The case $j_1=\frac{1}{2}$ follows from Proposition \ref{prop:R^1/2}. 
		
		\medskip
		\noindent Assume the result holds for $j_1$. We will use \eqref{eq:Rdef2} to show that the result holds for $j_1+\frac{1}{2}$. We view all the matrices appearing in \eqref{eq:Rdef2} as block matrices where each block is of size $(2j_2+1) \times (2j_2+1)$. In particular, we write 
		\begin{equation*}
			R^{(\frac{1}{2},j_2)}(q^{-j_1}t)=
			\begin{pmatrix}
				R^{(\frac{1}{2},j_2)}_{(1,1)}(q^{-j_1}t) & R^{(\frac{1}{2},j_2)}_{(1,2)}(q^{-j_1}t)\\
				R^{(\frac{1}{2},j_2)}_{(2,1)}(q^{-j_1}t) & R^{(\frac{1}{2},j_2)}_{(2,2)}(q^{-j_1}t)
			\end{pmatrix}, 
		\end{equation*}
		\begin{equation*}
			R^{(j_1,j_2)}(q^{\frac{1}{2}}t)=
			\begin{pmatrix}
				R^{(j_1,j_2)}_{(1,1)}(q^{\frac{1}{2}}t) & \cdots & R^{(j_1,j_2)}_{(1,2j_1+1)}(q^{\frac{1}{2}}t)\\
				\vdots & & \vdots\\
				R^{(j_1,j_2)}_{(2j_1+1,1)}(q^{\frac{1}{2}}t) & \cdots & R^{(j_1,j_2)}_{(2j_1+1,2j_1+1)}(q^{\frac{1}{2}}t)
			\end{pmatrix}. 
		\end{equation*}
		
		\noindent Note that 
		\begin{align*}
			&R^{(\frac{1}{2},j_2)}_{13}(q^{-j_1}t) R^{(j_1,j_2)}_{23}(q^{\frac{1}{2}}t)\\
			&=
			\begin{pmatrix}
				\mI_{2j_1+1} \otimes R^{(\frac{1}{2},j_2)}_{(1,1)}(q^{-j_1}t) & \mI_{2j_1+1} \otimes R^{(\frac{1}{2},j_2)}_{(1,2)}(q^{-j_1}t)\\
				\mI_{2j_1+1} \otimes R^{(\frac{1}{2},j_2)}_{(2,1)}(q^{-j_1}t) & \mI_{2j_1+1} \otimes R^{(\frac{1}{2},j_2)}_{(2,2)}(q^{-j_1}t)
			\end{pmatrix}
			\begin{pmatrix}
				R^{(j_1,j_2)}(q^{\frac{1}{2}}t) & 0\\
				0 & R^{(j_1,j_2)}(q^{\frac{1}{2}}t)
			\end{pmatrix}. 
		\end{align*}
		
		\noindent By \eqref{eq:Rdef2} and Definitions \ref{def:E}, \ref{def:F}, the $(a,b)$-block of $R^{(j_1+\frac{1}{2},j_2)}(t)$ is equal to a linear combination of the terms 
		\begin{align*}
			&R^{(\frac{1}{2},j_2)}_{(1,1)}(q^{-j_1}t)R^{(j_1,j_2)}_{(a,b)}(q^{\frac{1}{2}}t), && R^{(\frac{1}{2},j_2)}_{(1,2)}(q^{-j_1}t)R^{(j_1,j_2)}_{(a,b-1)}(q^{\frac{1}{2}}t),\\
			&R^{(\frac{1}{2},j_2)}_{(2,1)}(q^{-j_1}t)R^{(j_1,j_2)}_{(a-1,b)}(q^{\frac{1}{2}}t), && R^{(\frac{1}{2},j_2)}_{(2,2)}(q^{-j_1}t)R^{(j_1,j_2)}_{(a-1,b-1)}(q^{\frac{1}{2}}t). 
		\end{align*}
		
		\noindent Therefore, by Proposition \ref{prop:R^1/2} and the inductive hypothesis, the result holds for $j_1+\frac{1}{2}$. 
	\end{proof}
	
	\section{Removal of variable restriction}
	\noindent Recall that in Remark \ref{rm:baseilhacnotation} we assigned fixed values to the variables $\bar{k}_+,\bar{k}_-$, while in \cite[Theorem 2.10]{baseilhac1} one of these two variables is free and nonzero. In this appendix, we show that we can indeed allow one free boundary parameter $k$ in our main result. 
	
	\begin{definition}\label{def:Diag}\rm
		For $j \in \frac{1}{2}\mN^+$, define the diagonal matrix 
		\begin{equation*}
			D^{(j)}=\diag(1,k,\ldots,k^{2j}). 
		\end{equation*}
	\end{definition}
	
	\begin{lemma}\label{lem:DDR}\rm
		For $j_1,j_2 \in \frac{1}{2}\mN^+$, 
		\begin{equation}\label{eq:DDR}
			\left[D^{(j_1)} \otimes D^{(j_2)},R^{(j_1,j_2)}(t)\right]=0. 
		\end{equation}
	\end{lemma}
	\begin{proof}
		We view $\left[D^{(j_1)} \otimes D^{(j_2)},R^{(j_1,j_2)}(t)\right]$ and $R^{(j_1,j_2)}(t)$ as block matrices where each block is of size $(2j_2+1) \times (2j_2+1)$. Then the $(a,b)$-block of $\left[D^{(j_1)} \otimes D^{(j_2)},R^{(j_1,j_2)}(t)\right]$ is equal to 
		\begin{align*}
			&k^{a-1}D^{(j_2)}R^{(j_1,j_2)}_{(a,b)}(t)-R^{(j_1,j_2)}_{(a,b)}(t)k^{b-1}D^{(j_2)}, 
		\end{align*}
		which is zero by Proposition \ref{prop:Rnonzero}. 
	\end{proof}
	
	\begin{definition}\label{def:bK}\rm
		For $j \in \frac{1}{2}\mN^+$, define the matrix 
		\begin{equation*}
			\bar{K}^{(j)}(t)=\left(D^{(j)}\right)^{-1}K^{(j)}(t)D^{(j)}. 
		\end{equation*}
	\end{definition}
	
	\begin{remark}\label{rm:RKRKalt}\rm
		The matrix $\bar{K}^{(\frac{1}{2})}(t)$ can be obtained from the matrix $K(u)$ in \cite[Theorem 2.10]{baseilhac1} up to a scalar multiple via the correspondence 
		\begin{align*}
			U &\mapsto t^{-2}, &&\\
			y_{n+1}^+ &\mapsto W_{-n}, & y_{-n}^+ &\mapsto W_{n+1},\\
			\widetilde{z}_{n+1}^+ &\mapsto q^{-1}(q^2-q^{-2})G_{n+1}, & z_{n+1}^+ &\mapsto q^{-1}(q^2-q^{-2})\tG_{n+1},\\
			\bar{k}_+ &\mapsto q^{-\frac{1}{2}}(q+q^{-1})^{-\frac{1}{2}}(q-q^{-1})k, & \bar{k}_- &\mapsto q^{-\frac{1}{2}}(q+q^{-1})^{-\frac{1}{2}}(q-q^{-1})k^{-1}.
		\end{align*}
	\end{remark}
	
	\begin{theorem}\label{thm:RKRKalt}\rm
		For $j_1,j_2 \in \frac{1}{2}\mN^+$, 
		\begin{equation}\label{eq:RKRKalt}
			R^{(j_1,j_2)}(t/s) \star \bar{K}^{(j_1)}_{1}(s) \star \hR^{(j_1,j_2)} \star \bar{K}^{(j_2)}_{2}(t)=\bar{K}^{(j_2)}_{2}(t) \star \hR^{(j_1,j_2)} \star \bar{K}^{(j_1)}_{1}(s) \star R^{(j_1,j_2)}(t/s). 
		\end{equation}
	\end{theorem}
	\begin{proof}
		By Lemma \ref{lem:DDR} and Definition \ref{def:bK}, we have 
		\begin{align*}
			&R^{(j_1,j_2)}(t/s) \star \bar{K}^{(j_1)}_{1}(s) \star \hR^{(j_1,j_2)} \star \bar{K}^{(j_2)}_{2}(t)\\
			&=R^{(j_1,j_2)}(t/s)\left(D^{(j_1)} \otimes D^{(j_2)}\right)^{-1} \star \left(K^{(j_1)}(s) \otimes \mI_{2j_2+1}\right)\\
			&\hspace{4em} \star \left(D^{(j_1)} \otimes D^{(j_2)}\right)\hR^{(j_1,j_2)}\left(D^{(j_1)} \otimes D^{(j_2)}\right)^{-1} \star \left(\mI_{2j_1+1} \otimes K^{(j_2)}(t)\right)\left(D^{(j_1)} \otimes D^{(j_2)}\right)\\
			&=\left(D^{(j_1)} \otimes D^{(j_2)}\right)^{-1}R^{(j_1,j_2)}(t/s) \star \left(K^{(j_1)}(s) \otimes \mI_{2j_2+1}\right)\\
			&\hspace{8em} \star \hR^{(j_1,j_2)} \star \left(\mI_{2j_1+1} \otimes K^{(j_2)}(t)\right)\left(D^{(j_1)} \otimes D^{(j_2)}\right). 
		\end{align*}
		
		\noindent Similarly we have 
		\begin{align*}
			&\bar{K}^{(j_2)}_{2}(t) \star \hR^{(j_1,j_2)} \star \bar{K}^{(j_1)}_{1}(s) \star R^{(j_1,j_2)}(t/s)\\
			&=\left(D^{(j_1)} \otimes D^{(j_2)}\right)^{-1}\left(\mI_{2j_1+1} \otimes K^{(j_2)}(t)\right) \star \left(D^{(j_1)} \otimes D^{(j_2)}\right)\hR^{(j_1,j_2)}\left(D^{(j_1)} \otimes D^{(j_2)}\right)^{-1}\\
			&\hspace{4em} \star \left(K^{(j_1)}(s) \otimes \mI_{2j_2+1}\right) \star \left(D^{(j_1)} \otimes D^{(j_2)}\right)R^{(j_1,j_2)}(t/s)\\
			&=\left(D^{(j_1)} \otimes D^{(j_2)}\right)^{-1}\left(\mI_{2j_1+1} \otimes K^{(j_2)}(t)\right) \star \hR^{(j_1,j_2)}\\
			&\hspace{8em} \star \left(K^{(j_1)}(s) \otimes \mI_{2j_2+1}\right) \star R^{(j_1,j_2)}(t/s)\left(D^{(j_1)} \otimes D^{(j_2)}\right). 
		\end{align*}
		
		\noindent By the above discussion and \eqref{eq:RKRK}, we obtain \eqref{eq:RKRKalt}. 
	\end{proof}
	
	\noindent We also have the following corollary. 
	
	\begin{corollary}\label{cor:KRKRalt}\rm
		For $j_1,j_2 \in \frac{1}{2}\mN^+$, 
		\begin{equation}\label{eq:KRKRalt}
			\bar{K}^{(j_1)}_{1}(s) \star \hR^{(j_1,j_2)} \star \bar{K}^{(j_2)}_{2}(t) \star R^{(j_1,j_2)}(s/t)=R^{(j_1,j_2)}(s/t) \star \bar{K}^{(j_2)}_{2}(t) \star \hR^{(j_1,j_2)} \star \bar{K}^{(j_1)}_{1}(s). 
		\end{equation}
	\end{corollary}
	\begin{proof}
		Multiply both sides of \eqref{eq:RKRKalt} by $R^{(j_1,j_2)}(s/t)$ on the left and on the right and simplify the result using Lemma \ref{lem:RR}. This yields \eqref{eq:KRKRalt}. 
	\end{proof}
	
	\section{Connection to the quasi R-matrix}
	\noindent In this appendix, we discuss an alternative approach to obtain a Freidel-Maillet type equation of the same form as \eqref{eq:KRKR}, using the quasi R-matrix originally introduced in \cite[Chapter 4]{lusztig2}. This approach is from the perspective of universal properties, while it does not yield a closed form for the K-matrix. 
	
	\medskip
	\noindent Let $\mathfrak{g}$ be a symmetrizable Kac-Moody Lie algebra and let $U_q=U_q(\mathfrak{g})$ be its quantized enveloping algebra \cite{drinfeld,jimbo}. The algebra $U_q$ has a triangular decomposition $U_q=U_q^- \otimes U_q^0 \otimes U_q^+$. Let $\cO^\infty$ denote the category of weight modules with a locally finite $U_q^+$-action \cite[Section 2.3]{AV25}. The category $\cO^\infty$ coincides with the category $\cC^{hi}$ in \cite[Section 3.4.7]{lusztig2}. When we say that $U_q$ is a quasitriangular bialgebra, it is intended as the statement that its category $\cO^\infty$ is a braided monoidal category. Let $R$ denote the universal R-matrix of $U_q$. We consider the completion $U_q^\text{c}$, which is the algebra of natural transformations from the forgetful functor $\cO^\infty \to \text{Vect}$ to itself; see \cite[Section 2.9]{AV22}, \cite[Section 3.1]{BK}. 
	
	\medskip 
	\noindent Next we recall a factorization of $R$. Let $\widetilde{R} \in (U^- \otimes U^+)^\text{c}$ denote the quasi R-matrix \cite[Section 4.1.4]{lusztig2}. As in \cite[Section 7.3]{jantzen} or \cite[Lemma 4.3.2]{tanisaki}, we define $q^\Omega \in (U^0 \otimes U^0)^\text{c}$ by $q^\Omega$ acting on $V_\lambda \otimes W_\mu$ as scalar multiplication by $q^{(\lambda,\mu)}$, then we have 
	\begin{equation}\label{eq:R=OR}
		R=q^\Omega \widetilde{R}. 
	\end{equation}
	
	\noindent We define $\xi=\{\xi_V\}_{V \in \cO^\infty} \in U^\text{c}$ by $\xi_V$ acting on $V_\lambda$ as scalar multiplication by $q^{(\lambda,\lambda)/2}$. By \cite[(4.8)]{AV22}, we have 
	\begin{equation}\label{eq:Dxi}
		\Delta(\xi)=(\xi \otimes \xi)q^\Omega, 
	\end{equation}
	where $\Delta$ denotes the coproduct. 
	
	\medskip
	\noindent Now we define a K-matrix 
	\begin{equation}\label{eq:K=XR}
		K=(1 \otimes \xi^{-1})\widetilde{R}. 
	\end{equation}
	
	\medskip
	\noindent We will construct a Freidel-Maillet type equation of the same form as \eqref{eq:KRKR} using $R$, $q^\Omega$, and $K$. 
	
	\medskip
	\noindent By \eqref{eq:R=OR} and \cite[(2.2)]{AV22}, we have 
	\begin{equation}\label{eq:id*D}
		(\text{id} \otimes \Delta)(\widetilde{R})=q^{-\Omega_{12}}\widetilde{R}_{13}q^{\Omega_{12}}\widetilde{R}_{12}. 
	\end{equation}
	
	\noindent On both sides of \cite[(2.1)]{AV22}, plug in $x=q^\Omega$ and take tensor product with $1$ on the left, then simplify the result using \cite[(2.2)]{AV22} and \eqref{eq:R=OR}. This yields 
	\begin{equation*}\label{eq:Rconj1}
		q^{-\Omega_{13}}\widetilde{R}_{23}q^{\Omega_{13}}=q^{\Omega_{12}}\widetilde{R}_{23}q^{-\Omega_{12}}. 
	\end{equation*}
	
	\noindent Apply a flip to the above equation on the first two legs. This yields 
	\begin{equation}\label{eq:Rconj2}
		q^{-\Omega_{23}}\widetilde{R}_{13}q^{\Omega_{23}}=q^{\Omega_{12}}\widetilde{R}_{13}q^{-\Omega_{12}}. 
	\end{equation}
	
	\noindent On both sides of \cite[(2.1)]{AV22}, plug in $x=K$ and take tensor product with $1$ on the left, then simplify the result using \eqref{eq:Dxi}--\eqref{eq:Rconj2}. This yields 
	\begin{equation}\label{eq:RKRKuni}
		R_{23}K_{13}q^{-\Omega_{23}}K_{12}=K_{12}q^{-\Omega_{23}}K_{13}R_{23}. 
	\end{equation}
	
	\noindent Evaluating \eqref{eq:RKRKuni} on the first leg, we obtain a Freidel-Maillet type equation of the same form as \eqref{eq:KRKR}. 
	
	\medskip
	\noindent We now specialize to the case $\widehat{\mathfrak{sl}}_2$. By shifting the grading of the universal R-matrix following \cite{drinfeld2}, it becomes a formal series with spectral parameter. Evaluating \eqref{eq:RKRKuni} with the second and third legs on tensor products of finite-dimensional irreducible modules, we obtain that $R_{23}$ becomes the fused R-matrix of Definition \ref{def:R}, up to a nonzero scalar multiple, and that $q^{-\Omega_{23}}$ becomes the $\widehat{R}$-matrix of Definition \ref{def:hR}, up to sign convention. Consequently, evaluating \eqref{eq:RKRKuni} on the first leg yields a Freidel-Maillet type equation of the same form as \eqref{eq:RKRK}, with the same R-matrix and the $\widehat{R}$-matrix as in Theorem \ref{thm:RKRK}; the only remaining point is to identify the evaluations of $K$. 
	
	\begin{remark}\label{rm:D1}\rm
		The element $K$ in \eqref{eq:K=XR} is expected to be recoverable from the universal K-matrix associated with the q-Onsager algebra in the framework of tensor K-matrices \cite{AV24}; see also \cite{LBG}. Indeed, starting from the two-legged universal K-matrix for the quantum symmetric pair whose coideal subalgebra is the q-Onsager algebra, and by a suitable parameter specialization, one arrives at \eqref{eq:K=XR}. This observation, together with the above discussions, suggests the following conjecture. 
	\end{remark}
	
	\begin{conjecture}\label{conj:D2}\rm
		For $j\in\frac12\mathbb{N}^+$, the matrix $K^{(j)}(t)$ of Definition \ref{def:K} is a Cartan correction of the quasi R-matrix $\widetilde{R}$ with its first leg acting on the $(2j+1)$-dimensional irreducible $U_q(\widehat{\mathfrak{sl}}_2)$-module; equivalently, $K^{(j)}(t)$ coincides, up to a scalar multiple and the normalization of Appendix C, with the evaluation of the element $K$ in \eqref{eq:K=XR} with its first leg acting on such a module. 
	\end{conjecture}
	
	\noindent We mention two pieces of evidence. First, by \eqref{eq:Krecur}, the K-matrices satisfy a fusion relation matching the fusion procedure for the K-operators in \cite[Definition 5.6]{LBG}. Second, the limit \eqref{eq:lim} is consistent with the factorization \eqref{eq:R=OR}: as $t\to\infty$, the entries of $R^{(j_1,j_2)}(t)$ are dominated by the Cartan factor, reflecting the fact that $\widetilde{R}$ lies in
	$(U_q^-\otimes U_q^+)^c$.

	\bigskip
	\noindent Chenwei Ruan \\
	Beijing Institute of Mathematical Sciences and Applications, Beijing 101408, China; \\
	Yau Mathematical Sciences Center, Tsinghua University, Beijing 100084, China. \\
	email: \href{mailto:ruanchenwei@bimsa.cn}{\nolinkurl{ruanchenwei@bimsa.cn}}; \href{mailto:cwruan@outlook.com}{\nolinkurl{cwruan@outlook.com}}
\end{document}